\documentclass[11pt]{amsart}

\usepackage[T1]{fontenc}
\usepackage{lmodern}
\usepackage{microtype}
\usepackage{amsmath,amssymb,amsthm,mathtools,mathrsfs}
\usepackage{booktabs}
\usepackage{enumitem}
\usepackage{xcolor}
\usepackage{geometry}
\usepackage{aliascnt}
\usepackage{hyperref}
\usepackage[nameinlink,capitalise,noabbrev]{cleveref}
\usepackage{array}

\hypersetup{
  colorlinks=true,
  linkcolor=blue!55!black,
  citecolor=blue!55!black,
  urlcolor=blue!55!black,
  pdfauthor={Gerald Hoehn},
  pdftitle={Definite integral Albert algebras: arithmetic and finite geometry}
}

\newtheorem{theorem}{Theorem}[section]
\theoremstyle{plain}
\newaliascnt{proposition}{theorem}
\newtheorem{proposition}[proposition]{Proposition}
\aliascntresetthe{proposition}
\crefname{proposition}{Proposition}{Propositions}
\newaliascnt{lemma}{theorem}
\newtheorem{lemma}[lemma]{Lemma}
\aliascntresetthe{lemma}
\crefname{lemma}{Lemma}{Lemmas}
\newaliascnt{corollary}{theorem}
\newtheorem{corollary}[corollary]{Corollary}
\aliascntresetthe{corollary}
\crefname{corollary}{Corollary}{Corollaries}
\newaliascnt{conjecture}{theorem}

\aliascntresetthe{conjecture}
\crefname{conjecture}{Conjecture}{Conjectures}
\theoremstyle{definition}
\newaliascnt{definition}{theorem}
\newtheorem{definition}[definition]{Definition}
\aliascntresetthe{definition}
\crefname{definition}{Definition}{Definitions}
\newaliascnt{construction}{theorem}

\aliascntresetthe{construction}
\crefname{construction}{Construction}{Constructions}
\theoremstyle{remark}
\newaliascnt{remark}{theorem}
\newtheorem{remark}[remark]{Remark}
\aliascntresetthe{remark}
\crefname{remark}{Remark}{Remarks}

\newcommand{\Z}{\mathbb Z}
\newcommand{\Q}{\mathbb Q}
\newcommand{\R}{\mathbb R}
\newcommand{\C}{\mathbb C}
\newcommand{\F}{\mathbb F}
\newcommand{\cO}{\mathcal O}

\newcommand{\Aut}{\operatorname{Aut}}
\newcommand{\Tr}{\operatorname{Tr}}
\newcommand{\Her}{\operatorname{Her}}
\newcommand{\Span}{\operatorname{span}}
\newcommand{\Stab}{\operatorname{Stab}}
\newcommand{\Mass}{\operatorname{Mass}}
\newcommand{\rk}{\operatorname{rk}}

\newcommand{\GL}{\operatorname{GL}}
\newcommand{\SO}{\operatorname{SO}}
\newcommand{\Tri}{\operatorname{Tri}}
\newcommand{\GO}{\operatorname{O}}

\newcommand{\id}{\mathrm{id}}

\title[Definite integral Albert algebras]{Definite integral Albert algebras:\\ arithmetic and finite geometry}
\author{Gerald H\"ohn}
\address{Department of Mathematics, Kansas State University, Manhattan, Kansas 66506, USA}
\email{gerald@monstrous-moonshine.de}
\date{September 2026}

\begin{document}

\begin{abstract}
We give a unified treatment of definite Albert algebras over $\Z$, with
exceptional Siegel--Weil identities as the main arithmetic tool.  Starting
from the standard octonion order alone, its first scalar theta coefficient
saturates the genus average and forces every remaining class to be rootless.
A cubic-ring coefficient of the $G_2$ theta lift then gives the mass of the
remaining scale-$2$ framed orders, also recovered from shell geometry.
A uniform Niemeier--Golay stabilizer bound saturates this mass, proving that
there is exactly one further class
and determining both automorphism orders, frame transitivity, and the $2457$
frames on $819$ rank-one elements.  The associated generalized hexagon
identifies the nonstandard group as ${}^3D_4(2):3$ only after its order has been
proved.  An independent weighted-theta argument gives the trace-zero
root-system alternative $\varnothing$ or $A_2E_8^3$.  The construction relates
positive masses of framed objects to exceptional Eisenstein coefficients.
\end{abstract}

\maketitle

\section{Introduction}

The aim of this paper is to understand the simple group ${}^3D_4(2)$ through
its exceptional configurations, and to explain their geometry and arithmetic
as parts of a single picture.  The model is the role of the Leech lattice in
understanding the Conway group $\mathrm{Co}_1$ \cite{ConwaySloane}, together
with the analogous treatment of the Hall--Janko group \cite{HohnHallJanko}.
Here the central configuration consists of $819$ points of the Cayley plane:
its angle relations form an association scheme, and its orthogonal triples
are the lines of the generalized hexagon of order $(2,8)$
\cite{Hoggar1989,ElkiesGrossCubic,Hoehn26}.  Its full geometric stabilizer is
${}^3D_4(2):3$, containing the simple group with index $3$; this is the group
recovered as the automorphism group of the exceptional integral Albert order.
We relate the configuration to
its rank-$26$ lattice, integral Albert algebra, local stabilizers, and theta
series.  The objective is to understand why these structures belong together,
not only to recover the name and order of their automorphism group.

The same viewpoint has guided the author's work on the Fischer groups
through the Conway--Parker algebra and on the Rudvalis group through the
Conway--Wales lattice: a distinguished algebra or lattice and its finite
configurations provide a way to understand the group as a whole.

The two definite integral Albert classes and their automorphism groups are
known from \cite{Gross,ElkiesGross,Conrad,GPR}.  The associated rootless
determinant-$3$ lattice was already characterized in Borcherds's Lorentzian
lattice analysis \cite[Chapter~5.7]{BorcherdsThesis}; it also occurs in Nebe's
classification of finite rational matrix groups and their invariant lattices
\cite{Nebe1996}.  Elkies--Gross connect cubic embeddings into the exceptional
Albert order with rank-$24$ lattices and Hilbert modular forms, including
the discriminants $49$ and $16$ treated below
\cite[Sections~2--7]{ElkiesGrossCubic}.  In the final section they outline a
Euclidean uniqueness proof through shell geometry and Niemeier reconstruction.
They also explain how the reconstruction choices should yield the
lattice automorphism-group order, but explicitly defer this calculation to a
future paper \cite[Section~9, concluding paragraph]{ElkiesGrossCubic}.
We are not aware of a subsequent publication carrying out that specific
reconstruction count.

King uses coefficients of Siegel Eisenstein series to compute masses with
prescribed root systems; his Example~7 combines the mass for the root system
$E_6$ in rank $32$ with the known automorphism order of the rank-$26$ lattice
to prove its uniqueness \cite[Example~7]{King}.  M\'egarban\'e's online
computations enumerate all $678$ classes in the larger quadratic genus of
even rank-$26$ determinant-$3$ lattices, together with their orthogonal-group
orders \cite{MegarbaneX26}.  This quadratic classification is distinct from
the classification of Albert orders, which retains the cubic structure.

The present coefficient-saturation proof supplies a determination of the kind
envisaged by Elkies--Gross, combining explicit local Niemeier--Golay data with
genus-wide exceptional theta identities.  Starting with the standard order
$J_I=\Her_3(\cO)$, where $\cO$ is Coxeter's integral octonion algebra, we leave
the rest of its genus unspecified.  Gan's exceptional Siegel--Weil identity
for $F_4\times G_2\subset E_{8,4}$ \cite{GanSiegelWeil} and its weight-$12$
scalar boundary for $F_4\times\mathrm{PGL}_2\subset E_{7,3}$ give two positive
coefficient identities.  The first determines the standard automorphism order
and forces rootlessness of every remaining class; the second gives the total mass of
framed residual classes.  The matching local stabilizer bound forces
uniqueness, transitivity, and the remaining automorphism order.  We also
recover the framed mass independently from the uniform shell theorem in
\cref{sec:bounds}, giving a second route to this final saturation.  The group
identifications come afterwards.

The Hall--Janko coefficient and mass arguments \cite{HohnHallJanko} and
Borcherds--Freitag--Weissauer's degree-$12$, weight-$12$ Siegel cusp form
\cite[Sections~1--2]{BFW} provide methodological models.  In the latter,
signed masses of finite configurations determine Niemeier theta combinations;
King also relates his mass identities to this construction
\cite[Remark~2]{King}.  Here positive masses of framed Albert orders turn
exceptional Eisenstein coefficients into completeness and exact stabilizers.
The analogy concerns configurations and stabilizers; our $G_2$ theta series
are not Siegel cusp forms.

Conrad's Example~7.4 exhausts the mass after identifying the integral-point
groups of the two compact $F_4$ models using the \emph{ATLAS}
\cite[Example~7.4]{Conrad}; his introduction also discusses theoretical
replacements for computer verification.  Using the ring-theoretic foundation
\cite{GPR}, we obtain the orders by coefficient saturation before identifying
the groups.

Here ``integral Albert algebra'' means an Albert algebra over $\Z$, not an
arbitrary lattice closed under a rational Jordan product; every definite
one belongs to the hyperspecial genus of $J_I$.  Complete Peirce--triality
data retain the rational cubic tensor, integral Peirce intersections, and
additive glue.  Their uniqueness below is a consequence of classification
and frame transitivity, not a classification of additive codes.

The local inputs from \cite{Hoehn26}, stated uniformly in
\cref{sec:bounds}, are shell geometry, line deletion, and the action on
oriented deletion data for arbitrary rootless rank-$26$ determinant-$3$
lattices.  Its global uniqueness, frame transitivity, and group orders are
not used in saturation.  The lattice--hexagon correspondence enters later,
for the incidence automorphisms.  The Hall--Janko paper and \cite{BFW} are
parallels, not dependencies.

Sections~\ref{sec:dictionary}--\ref{sec:bounds} develop the genus, theta
coefficients, and local bound; Section~\ref{sec:saturation} completes the
classification and derives its arithmetic consequences, including the
cubic-field example in \cref{subsec:cubic-field49}.
Section~\ref{sec:venkov} gives an independent Venkov argument, and
Section~\ref{sec:groups} treats finite geometry and identifies the groups.
Appendix~\ref{sec:PTcode} records complete Peirce--triality reconstruction.

The main result records the automorphism orders as conclusions of these
arguments.  In the proof, the second class is denoted by $J_E$ only after
uniqueness; \cref{subsec:exotic-model} then identifies an explicit Elkies--Gross
representative.

\begin{theorem}[classification by coefficient saturation]\label{thm:main}
The definite Albert algebras over $\Z$ are $J_I$ and $J_E$, up to isomorphism.
Their automorphism orders are
\begin{equation}\label{eq:orders-intro}
 \begin{aligned}
 A:=|\Aut(J_I)|&=2^{15}3^6 5^2 7=4\,180\,377\,600,\\
 B:=|\Aut(J_E)|&=2^{12}3^5 7^2 13=634\,023\,936.
 \end{aligned}
\end{equation}
The group $\Aut(J_E)$ is transitive on its $2457$ unordered primitive scale-$2$
Jordan frames, with stabilizer of order $258048$.  The complete Peirce--triality
gluing data associated with these frames form one isomorphism class.  Every
admissible definite Peirce--triality datum with a primitive scale-$2$ frame and no
trace-$1$ rank-one element is isomorphic to these data.
\end{theorem}

The proof determines the two orders in different steps.  Triality first gives
$g_I:=|\Aut(J_I)|\le6|W(E_8)|$.  If $a(J)$ counts positive rank-one elements of
trace $1$, the scalar coefficient identity is
\[
 \sum_{[J]}\frac{240a(J)}{|\Aut(J)|}=\frac{120}{|W(E_8)|}.
\]
Since $a(J_I)=3$, its contribution already attains the right-hand side.
Thus the standard bound is sharp and every remaining class is rootless.
Each remaining class has a scale-$2$ frame.  The coefficient indexed by
$\Z+2\Z^3$ gives their total framed mass as $1/258048$; the uniform local
bound $|\Aut(J,\mathcal F)|\le258048$ forces a single framed class.
Only then does the reciprocal of the residual Gross mass become the order
of its automorphism group.  Frame transitivity and the count $2457$ follow.

For the recovered hexagon $H$, the mass argument also gives
\[
 \Aut(J_E)\cong\Aut^+(J_E^0)\cong\Aut(H),
 \qquad \Aut(J_E^0)\cong C_2\times\Aut(J_E).
\]
The conventional descriptions $2^2\!\cdot O_8^+(2)\!\cdot S_3$ and
${}^3D_4(2):3$ are obtained in the final main section, by octonionic triality
and the Steinberg--Tits hexagon.

A complementary argument derives the following alternative without the
mass or a prior enumeration of classes.

\begin{theorem}[exceptional Venkov dichotomy]\label{thm:venkov-intro}
Let $J$ be a definite integral Albert algebra in the hyperspecial genus.  Its
trace-zero lattice $J^0$ is even of rank $26$ and determinant $3$, and
\[
 R(J^0)=\varnothing\quad\text{or}\quad J^0\cong A_2\perp E_8^3.
\]
In the rooted case there are exactly three positive rank-one elements of trace
$1$, forming a Jordan frame, and $720$ primitive positive rank-one elements of
trace $2$.  In the rootless case there are no trace-$1$ rank-one elements and
there are $819$ primitive positive rank-one elements of trace $2$.
\end{theorem}

Shan's class-supported weighted theta correspondence and
$S_{14}(\mathrm{SL}_2(\Z))=0$ prove this independently of the two-class
classification \cite{ShanTheta}.

\section{Integral Albert algebras, frames, and mass}\label{sec:dictionary}

\subsection{The integral genus and its stabilizers}

For an Albert algebra $J$ over a commutative ring $R$, the automorphism functor
$\mathbf H_J=\underline{\Aut}_R(J)$ is an adjoint semisimple group scheme of type $F_4$.
Descent identifies Albert algebras with forms of this group scheme
\cite[Lemma~9.1 and Theorem~9.3]{GPR}.  In particular,
\[
 \mathbf H_J(\Z)=\Aut_{\Z}(J).
\]
For definite $J$, the real group $\mathbf H_J(\R)$ is compact, while its integral
points are a finite arithmetic subgroup.  These are the finite groups occurring in
the mass, not the compact real Lie group itself.

Let $G=\Aut(J_I\otimes\Q)$.  The rational classification by real forms and the
local splitting theorem imply that every definite Albert algebra over $\Z$ has
rational algebra $J_I\otimes\Q$ and is locally isomorphic to $J_I$ at every prime
\cite[Example~11.2 and Proposition~11.3]{GPR}.  Consequently all such algebras belong
to the single genus
\[
 \mathscr X=G(\Q)\backslash G(\mathbb A_f)/K_f,
 \qquad K_f=\prod_p\mathbf H_{J_I}(\Z_p).
\]
This uses the rational and local classifications, not the integral class-number
statement to be recovered below.

\begin{proposition}[mass stabilizers]\label{prop:mass-stabilizers}
The arithmetic stabilizer of the class corresponding to $J$ is $\Aut_{\Z}(J)$.
Its contribution to the mass is therefore $1/|\Aut_{\Z}(J)|$.
\end{proposition}

\begin{proof}
For a representative $g\in G(\mathbb A_f)$, the stabilizer is
$G(\Q)\cap gK_fg^{-1}$.  This is the group of rational Albert automorphisms preserving
the corresponding lattice at every prime, hence preserving the integral Albert
algebra itself.
\end{proof}

We use $E$, $N$, $x\mapsto x^{\#}$, $\Tr$, and $T$ for the identity, cubic norm,
quadratic adjoint, trace, and trace pairing of $J$.  Thus
\[
 \Tr(E)=T(E,E)=3,\qquad T(x,y)=\Tr(x\circ y)
\]
over $\Q$.  Integral statements are understood in the quadratic Jordan sense, so
no division by $2$ in the bilinear product is imposed over $\Z$.
For example, squares are integral by
\[
 x^2=x^{\#}+\Tr(x)x-S(x)E,
 \qquad S(x)=\tfrac12\bigl(\Tr(x)^2-T(x,x)\bigr)\in\Z.
\]
Local equivalence with $J_I$ makes $(J,T)$ unimodular, $E$ characteristic, and
$J^0=E^\perp$ even of rank $26$ and determinant $3$.

\subsection{Frames and triality}

A rational Jordan frame consists of primitive orthogonal idempotents
$e_1,e_2,e_3\in J\otimes\Q$ with $e_1+e_2+e_3=E$.  The associated Peirce decomposition is
\begin{equation}\label{eq:Peirce}
 J\otimes\Q=D\oplus V_1\oplus V_2\oplus V_3,
 \qquad D=\bigoplus_i\Q e_i,\quad \dim V_i=8,
\end{equation}
where $V_i=J_{jk}$ for $\{i,j,k\}=\{1,2,3\}$.  The cubic norm has the form
\begin{equation}\label{eq:triality-cubic}
\begin{split}
 N\left(\sum_i a_i e_i+\sum_i x_i\right)
 ={}&a_1a_2a_3-a_1q_1(x_1)-a_2q_2(x_2)-a_3q_3(x_3)\\
 &+\tau(x_1,x_2,x_3),\qquad x_i\in V_i.
\end{split}
\end{equation}
The three quadratic forms and the trilinear triality tensor $\tau$ encode the
rational product \cite{SpringerVeldkamp}.  The integral order may still glue together
the four rational summands; this is the information retained in \cref{sec:PTcode}.

A \emph{scaled integral frame of scale $m$} is a set
$\mathcal F=\{t_1,t_2,t_3\}$ with $t_i=me_i\in J$.  It is primitive of scale $m$ if
$m$ is the least positive integer for which all $me_i$ are integral.  All frame
stabilizers and frame counts below are setwise and unordered.  For $m=1,2$, let $F_m(J)$ be the number of unordered rational Jordan frames
whose three multiples $me_i$ belong to $J$.  Thus $F_2$ includes scale-$1$
frames; on a class without trace-$1$ rank-one elements, all its scale-$2$
frames are primitive.

\subsection{Elementary rank-one facts}

Put
\[
 X_1(J)=\{e\in J:e\ge0,\ e^{\#}=0,\ \Tr(e)=1\},
 \qquad a(J)=|X_1(J)|.
\]
Write $R(J^0)=\{r\in J^0:T(r,r)=2\}$ for the ordinary roots.
Here positivity is with respect to the cone of squares in the Euclidean
Albert algebra $J\otimes\R$.  Its spectral theorem shows that every element
of $X_1(J)$ is a primitive idempotent.

\begin{lemma}[positive orthogonality]\label{lem:positive-orthogonality}
If $e,f$ are idempotents in a Euclidean Albert algebra and $T(e,f)=0$, then
$e\circ f=0$.
\end{lemma}

\begin{proof}
Write $f=f_1+f_{1/2}+f_0$ in the Peirce decomposition relative to $e$.
Associativity of the trace pairing and $f^2=f$ give
\[
 0=T(e,f)=T(e,f^2)=T(e\circ f,f)
   =T(f_1,f_1)+\tfrac12T(f_{1/2},f_{1/2}).
\]
Positivity forces $f_1=f_{1/2}=0$, and hence $e\circ f=0$.
\end{proof}

\begin{lemma}[a root produces a trace-one idempotent]\label{lem:root-to-idempotent}
If $r\in R(J^0)$, then
\[
 e=E-r^2
\]
belongs to $X_1(J)$.
\end{lemma}

\begin{proof}
Let $\lambda_1,\lambda_2,\lambda_3$ be the real Jordan eigenvalues of $r$.  Then
\[
 \lambda_1+\lambda_2+\lambda_3=0,
 \qquad
 \lambda_1^2+\lambda_2^2+\lambda_3^2=2.
\]
A one-variable Lagrange-multiplier calculation gives
\[
 |N(r)|=|\lambda_1\lambda_2\lambda_3|
 \le \frac{2}{3\sqrt3}<1.
\]
Since $N(r)\in\Z$, it follows that $N(r)=0$.  Moreover
\[
 S(r)=\frac12\bigl(\Tr(r)^2-T(r,r)\bigr)=-1.
\]
The cubic Cayley--Hamilton identity therefore reduces to
\[
 r^3-r=0.
\]
Consequently $e=E-r^2$ satisfies $e^2=e$ and
$\Tr(e)=3-T(r,r)=1$.  The adjoint identity gives $r^{\#}=r^2-E$, so
$e=-r^{\#}\in J$ by integrality of the quadratic adjoint.  An idempotent of
trace $1$ is primitive and positive.
\end{proof}

\subsection{The standard class and its elementary bound}\label{subsec:standard}

Let $J_I=\Her_3(\cO)$ for Coxeter's maximal order in the positive-definite
rational octonion algebra.  Its trace-zero lattice is
\begin{equation}\label{eq:standard-trace-lattice}
 J_I^0\cong A_2\perp E_8\perp E_8\perp E_8,
 \qquad |R(J_I^0)|=6+3\cdot240=726.
\end{equation}
The diagonal trace-zero part is $A_2$ and the three off-diagonal Peirce
lattices are copies of the octonion norm lattice $E_8$.

A positive rank-one element $e$ of trace $1$, with diagonal entries
$a_i\in\Z$ and off-diagonal entries $c_i\in\cO$, satisfies
\[
 1=\Tr(e)^2-2S(e)=T(e,e)=\sum_i a_i^2+2\sum_i n(c_i).
\]
All octonion norms are nonnegative integers.  Thus the off-diagonal entries
vanish and precisely one diagonal entry is $\pm1$; the trace forces the
positive sign.  The three diagonal idempotents are therefore the only
elements of $X_1(J_I)$, and
\begin{equation}\label{eq:standard-a}
 a(J_I)=3,\qquad F_1(J_I)=1.
\end{equation}

\begin{proposition}\label{prop:standard-bound}
For $J_I=\Her_3(\cO)$,
\[
 |\Aut(J_I)|\le 6|W(E_8)|=2^{15}3^6 5^2 7.
\]
\end{proposition}

\begin{proof}
The coordinate calculation above shows that the three
diagonal idempotents $e_1,e_2,e_3$ are the only positive rank-one elements of
trace $1$.  Therefore every automorphism
permutes them, and the kernel of the resulting map to $S_3$ is the integral triality group
\[
 \Tri(\cO)
 =\{(\alpha,\beta,\gamma)\in\SO(\cO)^3:
      \alpha(x)\beta(y)=\gamma(xy)\}.
\]
The three actions are proper isometries: the rational pointwise frame stabilizer is the triality group of type $D_4$, acting through its three eight-dimensional representations \cite{SpringerVeldkamp}.  Projection to the third component has kernel of order at most $2$.  Indeed, if $\gamma=\id$, then for $r=\beta(1)$ one obtains
\[
 \alpha(x)=xr^{-1},
 \qquad
 \beta(y)=ry,
 \qquad
 z(ry)=(zr)y\quad\text{for all }y,z.
\]
The last identity puts $r$ in the middle nucleus of the octonion algebra, which consists of scalars \cite{ConwaySmith}.  Hence $r=\pm1$.  Consequently
\[
 |\Tri(\cO)|
 \le2|\SO(E_8)|
 =|\GO(E_8)|
 =|W(E_8)|.
\]
The order of the Weyl group follows from its invariant degrees:
\[
 |W(E_8)|
 =2\cdot8\cdot12\cdot14\cdot18\cdot20\cdot24\cdot30
 =2^{14}3^5 5^2 7.
\]
Multiplication by $|S_3|=6$ gives the asserted bound.
\end{proof}

\subsection{Gross's mass and framed objects}\label{sec:mass-assisted}

For the compact rational group $G$, split at every finite place, Gross's
formula is \cite[(5.1)]{Gross}\cite[(7.1) and Example~7.4]{Conrad}
\begin{equation}\label{eq:F4mass}
 M:=\Mass(F_4)=\sum_{[J]\in\mathscr X}\frac1{g_J}
 =\prod_{d\in\{2,6,8,12\}}\frac{\zeta(1-d)}2
 =\frac{691}{D_0},
 \quad g_J=|\Aut(J)|,
\end{equation}
where $D_0=2^{15}3^6 5^2 7^2 13$.  This gives the total mass without
specifying how many classes contribute.  To separate them, we will also use
the mass of framed objects.

\begin{proposition}[coefficient-saturation principle]\label{thm:mass-assisted-conway}
Let $\mathscr Y$ be a nonempty collection of classes in a finite arithmetic
genus, and let $\mathcal F(J)$ be a finite nonempty invariant set of frames
on every $J\in\mathscr Y$.  Then
\[
 \sum_{[J]\in\mathscr Y}\frac{|\mathcal F(J)|}{|\Aut(J)|}
 =\sum_{[(J,\mathcal F)]}\frac1{|\Aut(J,\mathcal F)|},
\]
where the second sum is over isomorphism classes of framed objects.
If every frame stabilizer has order at most $s$ and this mass is $1/s$,
there is precisely one framed isomorphism class.  Consequently $\mathscr Y$
contains one class, its automorphism group is frame-transitive, and its
frame stabilizers have order $s$.
\end{proposition}

\begin{proof}
For each $J$, partition its frame set into automorphism orbits and apply
orbit--stabilizer.  This gives the equality.  Every summand on its right is
at least $1/s$, so a total of $1/s$ consists of exactly one summand, of that
value.  The nonemptiness of every frame set excludes undetected unframed
classes.
\end{proof}

The analogous rank-$3$ icosian calculation relates $525$ orthogonal frames,
a stabilizer bound of $2304$, and the mass giving order $1\,209\,600$
\cite[Section~6.2, Proposition~6.1 and Remark~6.2]{HohnHallJanko}.
Here an exceptional $G_2$ coefficient will supply the framed mass directly.

\section{Exceptional theta coefficients and the residual genus}\label{sec:coefficients}

\subsection{Class-supported scalar theta series}\label{sec:rank-one-theta}

The exceptional scalar theta series counts positive rank-one elements by
trace, rather than all vectors by quadratic norm:
\[
 \vartheta_J(\tau)=1+240
 \sum_{\substack{0\ne x\in J,\ x\ge0\\x^{\#}=0}}
       \sigma_3(c(x))q^{\Tr(x)},\qquad q=e^{2\pi i\tau}.
\]
Here $c(x)$ is the largest positive integer $d$ with $x\in dJ$, and
$\sigma_j(d)=\sum_{b\mid d}b^j$.

\begin{lemma}[class-supported scalar modularity]\label{lem:scalar-theta-local}
For every $J$ in the hyperspecial genus, $\vartheta_J$ belongs to
$M_{12}(\mathrm{SL}_2(\Z))$, independently of the class number and the mass.
\end{lemma}

\begin{proof}
We first justify transport to an arbitrary class without using the
integral classification.  Realize $J$ in $V=J_I\otimes\Q$ as
$V\cap g(J_I\otimes\widehat{\Z})$, with $g\in G(\mathbb A_f)$.

Let $\mathbf S$ be the norm-isometry group scheme of $(J_I,N)$.
It is simply connected of type $E_6$ and its real fiber has rank $2$
\cite[Proposition~6.5 and Section~7, first paragraph]{Conrad}.
Strong approximation away from infinity therefore gives
\[
 g=\delta_f k,\qquad
 \delta\in\mathbf S(\Q),\quad k\in\mathbf S(\widehat{\Z}),
 \qquad \delta J_I=J.
\]
Here $\delta_f$ is the finite adelic component of $\delta$.  This is a
norm-lattice isomorphism, not necessarily an Albert isomorphism fixing $E$.
Let $\delta^*$ be its contragredient for the trace pairing, characterized by
$T(\delta x,\delta^*y)=T(x,y)$; it also belongs to $\mathbf S(\Q)$
\cite[Remark~2.3.7]{ShanTheta}.  Self-duality of $J_I$ and $J$ gives
\[
 \delta^*J_I=(\delta J_I)^\vee=J^\vee=J,
 \qquad T(x,\delta^{-1}E)=\Tr(\delta^*x).
\]
The real norm-isometry group preserves the positive cone, and the norm
and its differential characterize Jordan rank one.  Thus reindexing by
$y=\delta^*x$ preserves positivity, rank one, and integral content; the
last identity supplies exactly the trace exponent for $J$.

Now use the undifferentiated exceptional theta kernel for
$F_4\times\mathrm{PGL}_2$ in $E_{7,3}$
\cite[Section~5.1, especially Remark~5.1.3]{ShanTheta}.
On $\mathscr X$ take the scalar algebraic modular form supported on $[g]$
with value $g_J=|\Aut(J)|$.  In the theta integral the single stabilizer
weight $1/g_J$ cancels this value.  Left rational automorphy and right
hyperspecial invariance replace the finite translate $g$ by
$\delta_\infty^{-1}$ at infinity.  The preceding contragredient
reindexing in the Fourier calculation gives precisely $\vartheta_J$,
including its constant term $1$.  The kernel gives weight $12$ at infinity
and level one at the finite places.  This proves the assertion without
knowing the number of classes or any stabilizer order.  The same transport
will be used for the linear-weighted calculation in
\cref{prop:class-supported-theta}.
\end{proof}

Let $a(J)$ be the number of positive rank-one elements of trace $1$.
Since $M_{12}(\mathrm{SL}_2(\Z))=\C E_{12}\oplus\C\Delta$, the constant term
and the coefficient of $q$ give
\begin{equation}\label{eq:theta-general-a}
 \vartheta_J
 =E_{12}+\left(240a(J)-\frac{65520}{691}\right)\Delta.
\end{equation}
\subsection{The genus-wide Siegel--Weil identity and its scalar boundary}
\label{subsec:genus-SW}

Use the dual pair $F_4\times G_2\subset E_{8,4}$, with $F_4$ compact at
infinity and $G_2$ split.  The level-one minimal theta kernel gives a
weight-$4$ quaternionic modular form $\Theta_J$ for each integral Albert
class $J$.  Normalize the minimal kernel so that a primitive integral
rank-one Freudenthal vector has coefficient $1$, as in the explicit formula
below.  This normalization agrees with that of Gan--Gross--Savin: for a
Gorenstein totally real cubic ring $C$ one has
\begin{equation}\label{eq:G2-embedding-coefficients}
 c_C(\Theta_J)=N(C,J),
\end{equation}
where $N(C,J)$ counts embeddings preserving the unit and cubic norm, not
orbits of embeddings \cite[Sections~4 and~10, Proposition~10.1]{GanGrossSavin}.
The coefficients are indexed by cubic rings, or equivalently by the
corresponding $\GL_2(\Z)$-orbits of integral binary cubic forms.
For their fixed exceptional Albert model, Elkies--Gross attach to an
individual cubic embedding its rank-$24$ orthogonal complement and a Hilbert modular form of parallel
weight $4$ \cite[Sections~2--4]{ElkiesGrossCubic}.  Here the order $C$ instead
indexes a coefficient counting embeddings into each genus member.

For later use we specify the extension to non-Gorenstein coefficients.
Identify $J^\vee$ with $J$ using its unimodular trace pairing, put
\[
 W_J(\Z)=\Z\oplus J\oplus J^\vee\oplus\Z,
 \qquad
 p_J(a,b,c,d)=a u^3+\Tr(b)u^2v+\Tr(c)uv^2+d v^3,
\]
and let $\operatorname{cont}(w)$ be the content in this integral Freudenthal
lattice.  After the fixed archimedean Whittaker factors have been removed,
the minimal-kernel formula is
\begin{equation}\label{eq:G2-lift-coefficient}
 c_f(\Theta_J)=
 \sum_{\substack{w\in W_J(\Z),\ \rk_F(w)=1\\p_J(w)=f}}
       \sigma_4(\operatorname{cont}(w)).
\end{equation}
Here $\rk_F$ denotes Freudenthal rank, distinct from Jordan rank.  This is
Pollack's explicit expansion, with the universal scalar chosen to agree
with \eqref{eq:G2-embedding-coefficients}
\cite[Theorem~1.0.1 and Section~2.5]{PollackMinimal}
\cite[discussion of $m=0$ following Theorem~1.2.1]{PollackTheta}.  All its terms are nonnegative
integers.  The transport established in the proof of
\cref{lem:scalar-theta-local} applies here as well: the norm group acts on
the $J$ and $J^\vee$ coordinates by the mutually contragredient maps
$\delta$ and $\delta^*$, respectively.  It therefore transports the full
integral Freudenthal lattice, preserving Freudenthal rank and content.
The projection $p_J$ is taken with the identity of $J$; it is not assumed
to be invariant under an arbitrary norm isometry.  Thus the same Fourier
calculation gives \eqref{eq:G2-lift-coefficient} on every class, without
an integral classification.

Let $\mathscr E$ be the spherical weight-$4$ Eisenstein series on $G_2$,
normalized by
\begin{equation}\label{eq:G2-normalization}
 c_{\Z^3}(\mathscr E)=1.
\end{equation}
This fixes a nondegenerate Fourier coefficient, not a constant term.

\begin{proposition}[full-genus theta averages]\label{prop:genus-SW}
There is a nonzero constant $\kappa$ such that
\begin{equation}\label{eq:G2-genus-average}
 \sum_{[J]\in\mathscr X}\frac{\Theta_J}{g_J}=\kappa\mathscr E.
\end{equation}
The scalar boundary satisfies
\begin{equation}\label{eq:scalar-genus-average}
 \sum_{[J]\in\mathscr X}\frac{\vartheta_J}{g_J}=M E_{12}.
\end{equation}
These are identities over the entire genus; neither assumes its class number
or the orders of its stabilizers.
\end{proposition}

\begin{proof}
Gan's exceptional Siegel--Weil theorem identifies the theta lift of the
constant function on $G(\Q)\backslash G(\mathbb A)$ with the weight-$4$
Eisenstein series on $G_2$ \cite[Theorem~15.5]{GanSiegelWeil}.
The arithmetic realization in terms of Albert theta series is developed by
Gan--Gross--Savin \cite[Sections~9--10]{GanGrossSavin}.
Pollack recalls the theta integral in \cite[Theorem~4.2.1]{PollackMinimal};
the paragraph following that theorem explicitly records Gan's full equality,
rather than only equality modulo cusp forms.
Normalize the measure so that each finite double coset has weight $1/g_J$.
The integral then gives \eqref{eq:G2-genus-average}, up to the still
unspecified common scalar $\kappa$.

To take the scalar boundary, first take the constant term along the
Heisenberg radical of $G_2$, then its holomorphic highest-weight component
on the $\GL_2$ Levi.  This operation commutes with integration over the
compact $F_4$ quotient.  On the minimal $E_8$ kernel it gives a fixed nonzero multiple of Kim's
holomorphic exceptional theta series, whose diagonal restriction to the
identity of $J$ is $\vartheta_J/240$
\cite[Theorem~1.0.1]{PollackMinimal}.  There are no extra nonzero rank-one
terms with zero projection: if $p_J(w)=0$ and $\rk_F(w)\le1$, the rank-one
equations give $a=d=0$, $b^{\#}=c^{\#}=0$, and zero traces of $b,c$;
definiteness implies $b=c=0$.  This is also the constant-term observation
in \cite[Claim~5.0.3 and its proof]{PollackTheta}.

For the spherical $G_2$ Eisenstein series the same holomorphic component is
a nonzero multiple of the elliptic Eisenstein series of weight $12$
\cite[Theorem~3.2.5 and Proposition~3.3.2]{PollackMinimal}.  More explicitly,
the constant-term formula here has only a character term and an Eisenstein
series on the Levi: the inducing datum is a character, not a cuspidal
$\GL_2$ form.  Its holomorphic component therefore has zero cuspidal
projection, so no $\Delta$-term occurs.  Its weight is $3\cdot4$, because
the diagonal Levi acts with the cubic automorphy factor.  Consequently
the scalar average is proportional to $E_{12}$.  Every
$\vartheta_J$ has constant term $1$, so comparison of constant terms makes
the proportionality factor exactly $M$.  This obtains
\eqref{eq:scalar-genus-average} from the full theta integral, rather than
from a two-class identity.
\end{proof}

\subsection{The first scalar coefficient isolates the standard class}

\begin{proposition}[first coefficient saturation]\label{prop:first-saturation}
The standard bound is attained:
\[
 A:=g_I=6|W(E_8)|=4\,180\,377\,600.
\]
Every other class has $a(J)=0$ and a rootless trace-zero lattice.  The
residual genus $\mathscr X_{\mathrm{res}}=\mathscr X\setminus\{[J_I]\}$ is
nonempty and has mass
\begin{equation}\label{eq:residual-mass}
 M_{\mathrm{res}}:=\sum_{[J]\in\mathscr X_{\mathrm{res}}}\frac1{g_J}
 =M-\frac1A=\frac{600}{D_0}
 =\frac1{634\,023\,936}.
\end{equation}
For every residual class,
\begin{equation}\label{eq:residual-theta}
 \vartheta_J=E_{12}-\frac{65520}{691}\Delta,
 \qquad |\mathcal R_2(J)|=819,
\end{equation}
where
\begin{equation}\label{eq:R2}
 \mathcal R_2(J)=\{t\in J:t\ge0,\ t^{\#}=0,\ \Tr(t)=2\}.
\end{equation}
All elements of $\mathcal R_2(J)$ are primitive.
\end{proposition}

\begin{proof}
Since $[q]E_{12}=65520/691$, \eqref{eq:scalar-genus-average} gives
\[
 \sum_{[J]}\frac{240a(J)}{g_J}
 =\frac{65520}{D_0}=\frac{720}{6|W(E_8)|}.
\]
By \cref{prop:standard-bound}, the standard contribution $720/g_I$ is at
least this total, and every other term is nonnegative.  Hence
$g_I=6|W(E_8)|$ and $a(J)=0$ on every remaining class.  Write $A=g_I$ for
the order now determined.  By \cref{lem:root-to-idempotent}, every residual
$J^0$ is rootless.  Since $D_0=91A$, subtracting $1/A$ from
\eqref{eq:F4mass} gives $M_{\mathrm{res}}=600/D_0>0$, proving both
\eqref{eq:residual-mass} and nonemptiness.

Substitute $a(J)=0$ into \eqref{eq:theta-general-a}.  Its coefficient of
$q^2$ is
\[
 \frac{65520}{691}\bigl(\sigma_{11}(2)+24\bigr)
 =\frac{65520}{691}(2049+24)=196560=240\cdot819.
\]
No rank-one trace-$2$ element is twice an integral one of trace $1$, so all
contents are $1$.  This proves the count.
\end{proof}

The standard series follows from $a(J_I)=3$ without a genus average\footnote{These
class assignments agree with \cite[Remark~5.1.3]{ShanTheta}.  The introductory
display in Section~1.2.3 of the cited arXiv version interchanges the two
assignments; the coefficient $240a(J_I)=720$ fixes the standard one.}:
\begin{equation}\label{eq:standard-theta}
 \vartheta_{J_I}=E_{12}+\frac{432000}{691}\Delta.
\end{equation}
Thus the arithmetic already gives the same extremal scalar series on every
residual class, without yet distinguishing their cubic structures.
We call an order \emph{extremal} if $a(J)=0$; by
\cref{cor:root-iff-frame}, proved independently in \cref{sec:venkov}, this is
equivalent to a rootless trace-zero lattice.

\subsection{Every residual class has a frame}

\begin{lemma}\label{lem:residual-frame-existence}
Every residual class has a primitive scale-$2$ frame, and its trace-zero
lattice has minimum $4$.
\end{lemma}

\begin{proof}
For distinct $t,u\in\mathcal R_2(J)$, positivity and Cauchy--Schwarz give
$0\le T(t,u)<4$.  The pairing is integral, and $T(t,u)=3$ would make $t-u$
a trace-zero root.  Hence $T(t,u)\in\{0,1,2\}$.

The $819$ elements lie on a sphere in the affine hyperplane $\Tr(y)=2$:
\[
 T\left(y-\frac23E,y-\frac23E\right)=\frac83.
\]
Restrictions of polynomials of degree at most $2$ to this sphere span a
space of dimension at most
$1+26+\binom{27}{2}-1=377$; the last subtraction is the sphere equation.
If no off-diagonal pairing were $0$, the polynomials
\[
 P_t(y)=(T(t,y)-1)(T(t,y)-2),\qquad t\in\mathcal R_2(J),
\]
would have values $P_t(u)=6\delta_{t,u}$ on the shell.  They would therefore
be $819$ independent functions in a space of dimension at most $377$.
There is an orthogonal pair $t_1=2e_1,t_2=2e_2$.
By \cref{lem:positive-orthogonality}, $e_1\circ e_2=0$; hence
$t_3=2E-t_1-t_2=2(E-e_1-e_2)$ completes a scale-$2$ frame in $J$.
It is primitive because $a(J)=0$.

The same argument with $P_t(y)=T(t,y)(T(t,y)-1)$ excludes the possibility
that all off-diagonal pairings belong to $\{0,1\}$.  Thus some pair has
$T(t,u)=2$, and $t-u\in J^0$ has norm $4$.  Rootlessness and evenness give
the opposite lower bound.
\end{proof}

This existence argument prevents any residual class from being invisible
to a frame coefficient.  It uses no line count or transitivity theorem.

\subsection{The cubic-ring coefficient for scale-\texorpdfstring{$2$}{2} frames}
\label{subsec:frame-coefficient}

Put
\[
 C_0=\Z^3,\qquad R_4=\Z[t]/(t^3-t),\qquad C_2=\Z+2\Z^3.
\]
Their discriminants are $1$, $4$, and $16$.  The subscript of $C_2$ denotes
its level inside the split order, not its discriminant.  The ring $C_2$ is
non-Gorenstein, so \eqref{eq:G2-embedding-coefficients} cannot be applied to
it without examining the lift formula.  Its interpretation through triples
of Jordan roots in the exceptional Albert order is already given in
\cite[Section~7, equation~(7.1)]{ElkiesGrossCubic}.  We require the following
weighted coefficient identity uniformly over the genus.

\begin{lemma}[the frame coefficient and its content correction]
\label{lem:G2-frame-coefficient}
For every definite integral Albert algebra in the genus,
\begin{equation}\label{eq:G2-frame-coefficients}
 c_{C_0}(\Theta_J)=6F_1(J),\qquad
 c_{C_2}(\Theta_J)=6F_2(J)+96F_1(J).
\end{equation}
In particular, on the residual genus $c_{C_2}(\Theta_J)=6F_2(J)$.
\end{lemma}

\begin{proof}
The form $uv(u+v)$ represents $C_0$.  Its lifts are $(0,e,f,0)$
with $e,f\in J$ of trace $1$.  The rank-one equations make them
orthogonal primitive idempotents, and the content is $1$.  The lifts
therefore count ordered integral Jordan frames, giving the first equality.
The binary cubic $2uv(u+v)$ represents $C_2$.  A lift in
\eqref{eq:G2-lift-coefficient} has the form $w=(0,b,c,0)$ with
$\Tr(b)=\Tr(c)=2$.  Its rank-one equations give
\[
 b^{\#}=c^{\#}=0,\qquad T(b,c)=0
\]
\cite{PollackLifting}.  By definiteness, $b/2$ and $c/2$ are positive
primitive idempotents; \cref{lem:positive-orthogonality} makes them Jordan
orthogonal.  Conversely, for an orthogonal pair the vector $(0,b,c,0)$ is
Freudenthal rank one.  This can be checked after putting the two idempotents
in a real Jordan frame: in its split diagonal cubic algebra the vector is
a decomposable tensor in the $2\times2\times2$ model.  The rank-one equations
are polynomial, so this verifies the original rational vector as well.
Thus lifts are in bijection with ordered scale-$2$ frames
$(b,c,2E-b-c)$.

A lift has content $1$ or $2$, since both traces are $2$.  Content $2$ means
that $b/2,c/2$, and hence $E-b/2-c/2$, are integral.  Such lifts are exactly
the $6F_1(J)$ ordered integral frames.  Their weight is
$\sigma_4(2)=17$ instead of $1$.  The total is therefore
$6F_2(J)+(17-1)6F_1(J)$.
\end{proof}

\begin{lemma}[the standard coefficients]\label{lem:standard-G2}
One has
\begin{equation}\label{eq:standard-G2-coefficients}
 c_{C_0}(\Theta_I)=6,\qquad
 c_{R_4}(\Theta_I)=726,\qquad
 c_{C_2}(\Theta_I)=2262.
\end{equation}
For every residual class, $c_{C_0}(\Theta_J)=c_{R_4}(\Theta_J)=0$.
\end{lemma}

\begin{proof}
The split coefficient follows from $F_1(J_I)=1$.
The form $uv(u+2v)$ represents $R_4$.  Its rank-one lifts have the
form $(0,e,c,0)$ with $\Tr(e)=1$, $\Tr(c)=2$, and $e,c/2$
orthogonal primitive idempotents.  Their content is $1$.  The map
$r=c+e-E$ identifies these lifts with $R(J^0)$: its inverse is
\[
 e=E-r^2,\qquad c=r^2+r.
\]
Indeed, a root has eigenvalues $-1,0,1$ by
\cref{lem:root-to-idempotent}, so these elements are integral and give the
required lift.  Thus $c_{R_4}(\Theta_J)=|R(J^0)|$, directly from the
minimal-kernel formula.  This gives $726$ for $J_I$ and $0$ on the residual
genus.

To count the standard scale-$2$ frames, a positive rank-one element of
trace $2$ has integral nonnegative diagonal entries.  It is either $2e_i$
or has diagonal a permutation of $(1,1,0)$ and one off-diagonal octonion
of norm $1$.  There are $240$ choices for that octonion.  A frame consisting
of three primitive such elements would have its three zero diagonal entries
in distinct positions.  Two of these elements then have trace pairing $1$:
their diagonal supports overlap once and their off-diagonal supports are
disjoint.  This contradicts orthogonality.  Hence every scale-$2$ frame
contains a doubled diagonal idempotent.

There is one doubled diagonal frame.  Every other frame contains exactly
one doubled diagonal idempotent; its two complementary elements have
opposite off-diagonal octonions in the remaining $2\times2$ corner.
Consequently
\[
 F_2(J_I)=1+3\cdot\frac{240}{2}=361.
\]
By \cref{lem:G2-frame-coefficient},
$c_{C_2}(\Theta_I)=6\cdot361+96=2262$.
\end{proof}

\begin{proposition}[the Eisenstein coefficient at discriminant $16$]
\label{prop:G2-Eisenstein-coefficient}
In the normalization \eqref{eq:G2-normalization},
\begin{equation}\label{eq:G2-Eisenstein-coefficients}
 c_{R_4}(\mathscr E)=121,\qquad c_{C_2}(\mathscr E)=16577,
 \qquad \kappa=\frac6A.
\end{equation}
\end{proposition}

\begin{proof}
The split coefficient of \eqref{eq:G2-genus-average} and
\cref{prop:first-saturation,lem:standard-G2} give $\kappa=6/A$.
The $R_4$ coefficient then gives $c_{R_4}(\mathscr E)=726/6=121$.

Let $T_7(p)$ be the normalized Hecke operator whose Satake transform
is the character of the seven-dimensional representation of the dual group.
For a totally real cubic ring $R$ of $p$-depth zero, meaning that its binary
cubic form is not divisible by $p$, put $R_1=\Z+pR$ and let $n_R$ be the
number of rings strictly between $R_1$ and $R$.  The depth-zero case of
\cite[Proposition~15.6, p.~159]{GanGrossSavin} states, for a modular form $f$
of even weight $k$, that
\[
 \begin{aligned}
 c_R(T_7(p)f)
  ={}&p^{k-1}\!\sum_{\substack{R\subset R'\\{}[R':R]=p}}c_{R'}(f)
       +p^{-1}(n_R-1)c_R(f)\\
    &+p^{-k}\!\sum_{R_1\subsetneq R'\subsetneq R}c_{R'}(f)
       +p^{1-2k}c_{R_1}(f).
 \end{aligned}
\]
The first sum is over integral overorders in $R\otimes\Q$; the second is
over actual intermediate subrings, each of index $p$ in $R$, not over
isomorphism classes.

At weight $4$, the unnormalized inducing character is
$|\det|^{-5}=\delta_P^{5/3}$ \cite[Section~9]{GanGrossSavin}.
Since $\delta_P=|\det|^{-3}$, normalized induction has character
$|\det|^{-7/2}$, with spherical Levi parameters $p^4,p^3$.
The seven-dimensional representation restricts to the sum of the trivial,
standard, dual standard, determinant, and inverse determinant representations
\cite[Section~13]{GanGrossSavin}.  Its eigenvalue is therefore
\[
 \lambda_7(p)=p^7+p^4+p^3+1+p^{-3}+p^{-4}+p^{-7}.
\]
For $R=C_0=\Z^3$ and $p=2$, one has $R_1=C_2$ and $n_R=3$:
the three intermediate rings identify one pair of coordinates modulo $2$,
so each is isomorphic to $R_4$.  The overorder sum is empty because $C_0$
is maximal.  Taking $k=4$ and $f=\mathscr E$ gives
\[
 \lambda_7(2)c_{C_0}(\mathscr E)
 =\frac{3-1}{2}c_{C_0}(\mathscr E)
   +\frac3{16}c_{R_4}(\mathscr E)
   +\frac1{128}c_{C_2}(\mathscr E).
\]
Since $\lambda_7(2)=19609/128$, the last coefficient is
\[
 19609-128-24\cdot121=16577.
\]
This computation uses the spherical Eisenstein parameter and the standard
coefficients, not the number of frames in any residual class.
\end{proof}

The value $121$ also has an ideal-zeta interpretation.  Write
$\zeta_{R_4}(s)=\sum_{\mathfrak a}[R_4:\mathfrak a]^{-s}$, summing over
finite-index ideals.  Since $R_4\cong\Z\times\Z[u]/(u^2-1)$, the only
nonmaximal local factor occurs at $2$.  Its quadratic factor at $2$ is
$\mathcal O_2=\{(x,y)\in\Z_2^2:x\equiv y\pmod2\}$.
A proper ideal has projections $2^a\Z_2,2^b\Z_2$ with $a,b\ge1$.
Multiplication by $(2,0)$ and $(0,2)$ shows that, after rescaling the two
coordinates, its image modulo $2$ is either $\F_2^2$ or the diagonal line.
Thus the ideals of index $2^n$, $n\ge1$, are exactly
\[
 \begin{array}{ll}
 2^a\Z_2\times2^b\Z_2,&a,b\ge1,\quad a+b=n+1,\\
 (2^a,2^b)\mathcal O_2,&a,b\ge1,\quad a+b=n.
 \end{array}
\]
There are $n$ of the first kind and $n-1$ of the second.  Together with the
unit ideal, this gives the ideal series $(1-t+2t^2)/(1-t)^2$, where
$t=2^{-s}$.
Consequently
\[
 \zeta_{R_4}(s)=\zeta(s)^3(1-2^{-s}+2^{1-2s}),\qquad
 120^3\zeta_{R_4}(-3)=1-8+128=121.
\]
Thus the coefficient at $R_4$ agrees with this independently computed
ideal-zeta value.

\begin{corollary}[mass of the residual framed genus]\label{cor:residual-frame-mass}
One has
\begin{equation}\label{eq:residual-frame-mass}
 \sum_{[J]\in\mathscr X_{\mathrm{res}}}\frac{F_2(J)}{g_J}
 =\frac{16200}{A}=\frac1{258048}.
\end{equation}
\end{corollary}

\begin{proof}
Take the $C_2$ coefficient in \eqref{eq:G2-genus-average}, whose scalar is now
$6/A$, and subtract the standard contribution.  By
\cref{lem:G2-frame-coefficient,lem:standard-G2,prop:G2-Eisenstein-coefficient},
\[
 \sum_{[J]\in\mathscr X_{\mathrm{res}}}\frac{6F_2(J)}{g_J}
 =\frac{6\cdot16577-2262}{A}=\frac{97200}{A}.
\]
Divide by $6$.  No residual automorphism order or frame count has been used.
\end{proof}

\section{The uniform local frame-stabilizer bound}\label{sec:bounds}

\subsection{The shell attached to an arbitrary residual class}

Let $J$ be any class in $\mathscr X_{\mathrm{res}}$, and set $L=J^0$.
By \cref{prop:first-saturation,lem:residual-frame-existence}, this lattice
is even of rank $26$, determinant $3$, and minimum $4$.  For
$t\in\mathcal R_2(J)$ put
\begin{equation}\label{eq:projection}
 x_t=t-\frac23E.
\end{equation}
These vectors have norm $8/3$, lie in $L^\vee$, and belong to the same
nonzero discriminant class $C_+$, since their pairwise differences lie in
$L$.  Write $X=(C_+)_{8/3}$.  The shell counts and inner-product
multiplicities below occur in \cite[Proposition~9.1]{ElkiesGrossCubic}, whose
Section~9 relates them to the hexagon.  We use the uniform local formulation
in \cite[Sections~2--3]{Hoehn26}, proved by harmonic theta identities and
integrality for an arbitrary such lattice, rather than a global uniqueness
conclusion.

\begin{theorem}[local rank-$26$ shell geometry]\label{thm:frame-theorem}
The minimum in $C_+$ is $8/3$ and $|X|=819$.  For each $x\in X$, the inner
products with the other vectors are
\[
 -\frac43,\quad\frac23,\quad-\frac13
 \qquad\text{with multiplicities}\qquad18,\quad288,\quad512.
\]
If $(x,y)=-4/3$, then $-x-y\in X$.  The zero-sum triples are the lines of a
generalized hexagon of order $(2,8)$; the three inner products correspond
respectively to point-graph distances $1,2,3$.
\end{theorem}

The injective map $t\mapsto x_t$ is a bijection onto $X$, since both sets
have $819$ elements.  We need the local neighbours and their gluing, rather
than a total line count, for the stabilizer bound.

\begin{proposition}\label{prop:lines-frames}
The zero-sum triples in $X$ are exactly the unordered primitive scale-$2$
Jordan frames of $J$ under \eqref{eq:projection}.
\end{proposition}

\begin{proof}
Write $t_i=2e_i$ for positive primitive idempotents.  Then
\[
 T(x_{t_i},x_{t_j})=T(t_i,t_j)-\frac43.
\]
Thus adjacency is equivalent to $T(e_i,e_j)=0$, which is Jordan
orthogonality by \cref{lem:positive-orthogonality}.  The third element is
$t_k=2E-t_i-t_j\in J$, with $x_{t_k}=-x_{t_i}-x_{t_j}$.
Conversely every scale-$2$ frame gives such a triple.  Primitivity follows
from $a(J)=0$.
\end{proof}

As a geometric check on \cref{cor:residual-frame-mass}, the preceding results
give nine lines through each of $819$ points and three points on each line.
Thus, for every residual class, $F_2(J)=819\cdot9/3=2457$, and
\eqref{eq:residual-mass} gives
\[
 \sum_{[J]\in\mathscr X_{\mathrm{res}}}\frac{F_2(J)}{g_J}
   =2457M_{\mathrm{res}}=\frac1{258048}.
\]
Equivalently, $6\cdot16577-2262=6\cdot2457\,A M_{\mathrm{res}}$.  This is a geometric
derivation of the same framed mass; the $G_2$ calculation obtains it before
invoking the local shell theorem.

\subsection{The oriented deletion datum}\label{subsec:binary-shadow}

Let $J$ be any residual class, and fix a primitive scale-$2$ frame
$\mathcal F=\{t_1,t_2,t_3\}$ and its line $\ell\subset X$.
The discriminant-$16$ complement and its $A_1^{24}$ Niemeier completion
are studied in \cite[Section~7]{ElkiesGrossCubic} and used in the
reconstruction outlined in its Section~9.  Here we retain the full marked
gluing to obtain a stabilizer bound valid for every residual class.
The rank-$27$ construction in \cite[Section~6]{Hoehn26} gives a particularly useful
quadratic description of its stabilizer.  In the Albert model its odd unimodular
lattice is the trace lattice $\Lambda_J=(J,T)$ itself: take
\[
 \xi=-E,\qquad u_i=t_i-E=x_{t_i}+\xi/3.
\]
The characteristic vector $\xi$ has norm $3$, and
\[
 u_1+u_2+u_3=\xi,\qquad
 S_{\ell}=\langle u_1,u_2,u_3\rangle_{\Z},\qquad
 \operatorname{Gram}(S_{\ell})=4I_3-J_3.
\]
Here $J_3$ in the Gram matrix is the all-one matrix.  Equivalently, $\Lambda_J$ is the
index-three gluing of $L\perp\Z \xi$ determined by $X$; the identity above fixes
which discriminant class is used.

In the orthonormal basis $f_i=(u_i-\xi)/2=e_i$ of $S_{\ell}\otimes\Q$, let
$I_{\ell}=\bigoplus_i\Z f_i$.  Then
\[
 S_{\ell}=\{(a_1,a_2,a_3)\in\Z^3:a_1\equiv a_2\equiv a_3\pmod2\},
 \qquad I_{\ell}/S_{\ell}\cong\F_2^2.
\]
Let $K=S_{\ell}^\perp\cap \Lambda_J$.  Write $D(L)=L^\vee/L$ for the discriminant group of an
integral lattice $L$.  We record precisely the deletion input needed here.

\begin{proposition}[line deletion]\label{prop:deletion-package}
The lattice $S_{\ell}$ is primitive in $\Lambda_J$.  Its complement $K$ is even of rank $24$,
determinant $16$, and minimum at least $4$.  The gluing of $S_{\ell}\perp K$ to $\Lambda_J$
determines an index-four completion
\[
 K\subset\Gamma\cong N(A_1^{24}).
\]
The $24$ root pairs of $\Gamma$ fall into three Golay octads
$\mathcal T=\{O_1,O_2,O_3\}$, canonically corresponding to the three points of $\ell$.
Writing $W=\Gamma/K\cong\F_2^2$, the quotient map
\begin{equation}\label{eq:lambda}
 \lambda:\Gamma\longrightarrow W
\end{equation}
labels roots in $O_i$ by $w_i\in W\setminus\{0\}$.  Writing $q_K(\bar y)=y^2\pmod{2\Z}$, the full gluing is an
anti-isometry
\begin{equation}\label{eq:phi}
 \phi:(D(S_{\ell}),q_{S_{\ell},\xi})\longrightarrow(D(K),q_K),
 \qquad q_{S_{\ell},\xi}(\bar s)=s^2-(s,\xi)\pmod{2\Z},
\end{equation}
so $q_K(\phi(\bar s))=-q_{S_{\ell},\xi}(\bar s)$; its restriction identifies
$I_{\ell}/S_{\ell}$ with $\Gamma/K$.  All these constructions
are equivariant for line-preserving isometries.  Moreover,
\begin{equation}\label{eq:deletion-decomposition}
 J^0\otimes\Q=U_{\ell}\perp(K\otimes\Q),
 \qquad U_{\ell}=\Span_{\Q}\{x_{t_1},x_{t_2},x_{t_3}\}.
\end{equation}
\end{proposition}

Apply the line-deletion theorem and its discriminant-gluing description
\cite[Theorem~6.2 and Proposition~6.5]{Hoehn26} to the lattice
$L=J^0$.  Those statements start from an arbitrary even rank-$26$ lattice
of determinant $3$ and minimum $4$, so they apply uniformly to every
residual class by \cref{lem:residual-frame-existence}.  Their proofs use the
local shell and the discriminant gluing, not uniqueness of the lattice or a
known automorphism group.
The completion is defined by the image of $I_{\ell}/S_{\ell}$ under the discriminant
anti-isometry.  In particular, it does not require a choice of signs for the roots.

The root pairs can be read directly from the frame.  A neighbour $t'$ of $t_i$
off $\ell$ has distances $2$ from $t_j,t_k$, so \cref{thm:frame-theorem} gives
$T(t',t_i)=0$ and $T(t',e_j)=T(t',e_k)=1$.
The orthogonality argument in \cref{lem:positive-orthogonality}, applied to
$t'/2$ and $e_i$, gives Jordan orthogonality.  The Peirce decomposition then yields
\[
 t'=e_j+e_k+y,\qquad y\in V_i,\qquad T(y,y)=4-2=2.
\]
Moreover, $t'-E=-e_i+y\in \Lambda_J$.  The gluing in \eqref{eq:phi} therefore gives
$y\in\Gamma$ and
$\bar y=\phi(-\bar e_i)=\phi(\bar e_i)\in\Gamma/K$, since
$\bar e_i\in I_{\ell}/S_{\ell}$ has order $2$.  Hence $\lambda(y)=w_i$.
The third point $t''=2E-t_i-t'=e_j+e_k-y$ gives the opposite root.
For fixed $i$, the map $t'\mapsto y$ is injective, so the $16$ off-frame neighbours
give eight root pairs in $V_i$.  These are precisely the octad $O_i$ of
\cref{prop:deletion-package}.

We call
\begin{equation}\label{eq:binary-shadow}
 \mathcal D_{\mathcal F}=(\Gamma,\mathcal T,\lambda,\epsilon),
 \qquad \epsilon=\phi,
\end{equation}
the \emph{oriented deletion datum}.  An automorphism may permute the octads and the
three nonzero elements of $W$ simultaneously; it must intertwine $\lambda$ and the
full gluing, using the induced permutation of the $u_i$.  The frame marking is
therefore already encoded by the root labels.

There are two kinds of gluing associated with the frame.  The complete Peirce
data retain the cubic order; the quadratic deletion datum
\eqref{eq:binary-shadow} retains the marked trace lattice and suffices for a
stabilizer bound.  No identification of the latter with a quotient of the Peirce
code is needed.

A Leech neighbour of $\Gamma$ can be constructed after choosing root signs.  Such a
polarization is not intrinsic to the frame and is not included in
$\mathcal D_{\mathcal F}$.  All independent root sign changes must be allowed when
computing its automorphisms.

\begin{lemma}[faithfulness]\label{lem:shadow-faithful}
Restriction gives an injection
\begin{equation}\label{eq:stabilizer-injection}
 \Aut(J)_{\mathcal F}\hookrightarrow\Aut(\mathcal D_{\mathcal F}).
\end{equation}
\end{lemma}

\begin{proof}
An Albert automorphism fixes $E$, preserves the trace lattice, and transports the
entire deletion construction.  An element in the kernel of the restriction map,
hence acting trivially on $\Gamma$, acts trivially on $W$, since $\lambda$ is
surjective.  It consequently fixes each frame element,
and hence fixes $U_{\ell}$ pointwise.  It also fixes $K\otimes\Q$ pointwise.
Equation~\eqref{eq:deletion-decomposition} shows that it is the identity on
$J^0\otimes\Q$.  As it fixes $E$, it is the identity on $J$.
\end{proof}

\subsection{The finite Golay calculation}

The Golay description gives
\[
 \Gamma=\left\{\frac12\sum_{j=1}^{24}a_jr_j:
 a_j\in\Z,\ (a_j\bmod2)\in\mathcal G_{24}\right\},\qquad r_j^2=2,
\]
and hence
\[
 \Aut(\Gamma)=\Sigma\rtimes M_{24},\qquad
 \Sigma\cong C_2^{24},\qquad \Sigma\triangleleft\Aut(\Gamma).
\]
Here $\Sigma$ is the group of independent root sign changes.  Indeed, its roots span the space, every isometry permutes their pairs, and the induced
coordinate permutation must preserve the Golay glue.  Each of the $24$ independent
root reflections preserves that glue.

The Mathieu group is transitive on trios.  There are $759$ octads and $15$ trios
through an octad, so there are $759\cdot15/3=3795$ trios.  The classical Golay-code
order $|M_{24}|=244823040$ therefore gives
\begin{equation}\label{eq:trio-stabilizer}
 |(M_{24})_{\mathcal T}|=64512=2^{10}3^2 7;
\end{equation}
see \cite[Chapter~11]{ConwaySloane}.

For clarity we also state the finite input which determines the correct subgroup
of this trio stabilizer.  For a quotient map with the indicated root labels, use
unimodularity of $\Gamma$ to represent each functional $\chi\lambda$ by a class
$t_\chi\in\Gamma/2\Gamma$:
\[
 (t_\chi,a)\equiv\chi(\lambda(a))\pmod2\qquad(a\in\Gamma).
\]
The admissibility condition is
\begin{equation}\label{eq:deletion-admissibility}
 t_\chi^2\equiv2\pmod4\qquad(0\ne\chi\in W^*).
\end{equation}
The maps from line deletion satisfy this condition.  Their compatible gluings are
exactly the anti-isometries in \eqref{eq:phi} extending the specified identification
on $I_{\ell}/S_{\ell}$.

\begin{lemma}[finite deletion-data action]\label{lem:finite-deletion-orbits}
Fix $(\Gamma,\mathcal T)$ and the root labels $w_i\in W\setminus\{0\}$.
For $g\in\Aut(\Gamma)_{\mathcal T}$ inducing $\pi_g\in S_3$ on the octads, let
$h_g\in\Aut(W)$ and $s_g\in\Aut(S_{\ell})$ be given by
$h_g(w_i)=w_{\pi_g(i)}$ and $s_g(u_i)=u_{\pi_g(i)}$.  The action on the
oriented data is
\begin{equation}\label{eq:twisted-deletion-action}
 g\cdot\lambda=h_g\circ\lambda\circ g^{-1},\qquad
 g\cdot\phi=\bar g\circ\phi\circ\bar s_g^{-1},
\end{equation}
where $\bar s_g$ acts on $D(S_{\ell})$ and
$\bar g:D(K)\to D(gK)$ is induced by $g$.  Thus
$\ker(g\cdot\lambda)=gK$, and the root labels remain fixed.

There are $2^{21}$ admissible quotient maps, each admitting two compatible
gluings.  On the resulting set $\mathscr D$ of $2^{22}$ oriented data, $\Sigma$
has eight orbits, each of size $2^{19}$.  The induced action of
$T_{\mathcal T}:=(M_{24})_{\mathcal T}$ is transitive on these eight sign orbits.
\end{lemma}

These are the finite Golay calculations of
\cite[Lemma~6.4 and Propositions~6.5 and~6.8]{Hoehn26}.
It is a statement about fixed rank-$24$ data, not about orbits of Albert frames.

The two gluing choices can be seen directly from $S_{\ell}$.  In the coordinates
$f_i$, its dual is
\[
 S_{\ell}^{\vee}
 =\left\{\tfrac12(m_1,m_2,m_3):m_i\in\Z,\
                    m_1+m_2+m_3\equiv0\pmod2\right\}.
\]
Consequently $D(S_{\ell})\cong(\Z/4\Z)^2$ and
$2D(S_{\ell})=I_{\ell}/S_{\ell}$, on which $q_{S_{\ell},\xi}$ vanishes.
For cyclic $(i,j,k)$, put $\delta_i=\overline{(f_j-f_k)/2}$.
Then $2\delta_i=\bar f_i$, $\sum_i\delta_i=0$, and
$q_{S_{\ell},\xi}(\delta_i)=1/2$.  The only classes doubling to $\bar f_i$ with
this value of $q_{S_{\ell},\xi}$ are $\pm\delta_i$; the sum-zero relation forces
a common sign.  Thus the kernel of the isometry action on $I_{\ell}/S_{\ell}$
is $\{\pm1\}$.  Coordinate permutations give the full $S_3$ on its three
nonzero classes, so
\[
 \operatorname{Isom}(D(S_{\ell}),q_{S_{\ell},\xi})
 \cong\langle-1\rangle\times S_3.
\]
The cited extension result supplies one compatible gluing.  Any other differs
by an isometry in this kernel, so the two choices are $\phi$ and $-\phi$.

Briefly, in a trio marking the six-dimensional kernel of root sign changes
on quotient maps consists of the words
\[
 (l+c_1,l+c_2,l+c_3),\qquad
 l\in\langle x,y,z\rangle,\quad c_i\in\F_2.
\]
Thus their orbits on quotient maps have size $2^{18}$.  The three admissibility
equations leave eight such orbits.  Two explicit Golay permutations join these
orbits; their action is computed in the cited result.  Finally, $-1$ fixes a quotient
map and interchanges its two compatible gluings, doubling each sign orbit on the
oriented data.  We use only these finite conclusions, not the global reconstruction
and uniqueness consequences drawn from them in the companion paper.

The eight-point action has a useful geometric description.  Let $\Omega$ be the
eight sign orbits, with the induced action from \eqref{eq:twisted-deletion-action}.
The same Golay calculation, with its refinement in
\cite[Proposition~6.12]{Hoehn26}, gives
\begin{equation}\label{eq:eight-orbit-structure}
 T_{\mathcal T}\cong 2^6:\bigl(L_3(2)\times S_3\bigr).
\end{equation}
The normal $2^6$ and the block-permuting $S_3$ act trivially on $\Omega$.
The quotient $L_3(2)\cong L_2(7)$ acts as on $\mathbf P^1(\F_7)$, so a point
stabilizer is $7:3$.  Consequently
\begin{equation}\label{eq:eight-orbit-stabilizer}
 (T_{\mathcal T})_\omega\cong
 2^6:\bigl((7:3)\times S_3\bigr),\qquad \omega\in\Omega.
\end{equation}
This description is derived from the Golay generators, not from recognition of
a parabolic subgroup in ${}^3D_4(2)$.

\begin{proposition}[the fixed-datum stabilizer]\label{prop:frame-stabilizer-bound}
For every primitive scale-$2$ frame $\mathcal F$ in $J$,
\[
 |\Aut(\mathcal D_{\mathcal F})|=258048,
 \qquad |\Aut(J)_{\mathcal F}|\le258048.
\]
The kernel of the root-pair permutation action on $\Aut(\mathcal D_{\mathcal F})$
is $C_2^5$.  Its image is the index-eight subgroup
$2^6:\bigl((7:3)\times S_3\bigr)$ of $(M_{24})_{\mathcal T}$.
\end{proposition}

\begin{proof}
Put $\mathcal D=\mathcal D_{\mathcal F}$.  By
\cref{lem:finite-deletion-orbits}, $\Sigma\rtimes T_{\mathcal T}$ is transitive
on the $2^{22}$ oriented data.  Orbit--stabilizer immediately gives
\[
 |\Aut(\mathcal D)|
 =\frac{2^{24}\cdot64512}{2^{22}}=258048.
\]
For the kernel and image, let $Q_{\mathcal D}$ be the stabilizer of the sign
orbit $\Sigma\mathcal D$ in $T_{\mathcal T}$.  The same lemma gives
\[
 |\Sigma_{\mathcal D}|=2^{24-19}=2^5,
 \qquad |Q_{\mathcal D}|=64512/8=8064.
\]
Normality of $\Sigma$ makes the action on its orbits well-defined.  The image of
$\Aut(\mathcal D)$ in the permutation group is exactly $Q_{\mathcal D}$: a permutation preserving $\Sigma\mathcal D$ can be corrected
by a sign change to fix $\mathcal D$.  Its structure is
\eqref{eq:eight-orbit-stabilizer}.  Therefore
\begin{equation}\label{eq:correct-deletion-sequence}
 1\longrightarrow\Sigma_{\mathcal D}
 \longrightarrow\Aut(\mathcal D)
 \longrightarrow Q_{\mathcal D}\longrightarrow1
\end{equation}
is exact and gives the stated kernel and image, with factorization
$258048=2^5\cdot8064$.  Apply \cref{lem:shadow-faithful} for the Albert stabilizer bound.  Notice that the reduction from $2^6$ to $2^5$ stabilizes one
of the two gluings; it is not a quotient by global sign inside $\Aut(\mathcal D)$.
\end{proof}

\section{Coefficient saturation and the two integral classes}\label{sec:saturation}

\begin{proof}[Proof of \cref{thm:main}]
The first coefficient has already determined the standard order $g_I=A$
and the positive residual mass $M_{\mathrm{res}}$ in
\cref{prop:first-saturation}.
Every residual class has a primitive scale-$2$ frame by
\cref{lem:residual-frame-existence}.  Reinterpret
\cref{cor:residual-frame-mass} by the framed mass identity of
\cref{thm:mass-assisted-conway}:
\[
 \sum_{\substack{[(J,\mathcal F)]\\{}[J]\in\mathscr X_{\mathrm{res}}}}
       \frac1{|\Aut(J,\mathcal F)|}=\frac1{258048}.
\]
The uniform bound in \cref{prop:frame-stabilizer-bound} makes every summand
at least $1/258048$.  Hence there is exactly one framed isomorphism class,
with stabilizer of order $258048$.  Since every residual class has a frame,
there is exactly one residual Albert class.  Denote it by $J_E$.
Its contribution to \eqref{eq:residual-mass} now gives its group order:
\[
 B:=|\Aut(J_E)|=M_{\mathrm{res}}^{-1}
 =\frac{D_0}{600}=2^{12}3^5 7^2 13=634\,023\,936.
\]
Consequently the two determined orders account for the total mass:
\begin{equation}\label{eq:mass-split}
 M=\frac1A+\frac1B,\qquad D_0=91A=600B.
\end{equation}
The framed uniqueness gives frame transitivity; orbit--stabilizer yields
\begin{equation}\label{eq:2457}
 F_2(J_E)=\frac{B}{258048}=2457.
\end{equation}

The explicit identification with the Elkies--Gross isotope is given in
\cref{subsec:exotic-model} below and is not used in this deduction.
Complete Peirce--triality data reconstruct a framed cubic order by
\cref{thm:reconstruction}.  An admissible primitive scale-$2$ datum giving
a definite order without a trace-$1$ rank-one element therefore reconstructs
$J_E$.  Frame transitivity proves the remaining assertions about these data.
\end{proof}

For their fixed exceptional model, Elkies--Gross already obtain transitivity
and $14742$ embeddings of $\Z+2\Z^3$ from the known root-triple action
\cite[Section~7]{ElkiesGrossCubic}; dividing by $6$ gives $2457$ unordered
frames.  The argument above instead obtains these conclusions by saturating
the framed mass over the initially unspecified residual genus.

\subsection{The explicit second order}\label{subsec:exotic-model}

The class just obtained has the familiar Elkies--Gross representative.
In their octonion basis put
\[
 \beta=\frac{-1+\varepsilon_1+\cdots+\varepsilon_7}{2}\in\cO,
 \qquad \beta^2+\beta+2=0,
 \qquad
 v=\begin{pmatrix}
 2&\beta&\bar\beta\\ \bar\beta&2&\beta\\ \beta&\bar\beta&2
 \end{pmatrix}.
\]
Here $N(v)=8-12+\Tr_{\cO}(\beta^3)=1$, since $\beta^3=2-\beta$.
The element is positive, and the isotope $J_I^{(v)}$ is a definite integral
Albert algebra with identity $v^{-1}=v^{\#}$.  Elkies--Gross establish that
its trace-zero lattice is rootless \cite{ElkiesGross}.  It is therefore
nonstandard and, by the classification above,
\[
 J_E\cong J_I^{(v)}.
\]
This explicit rootlessness result is used only to name a representative of
the residual class; rootlessness of every residual class was already forced
by the scalar coefficient.

\begin{proposition}[the two trace-zero lattices]\label{prop:trace-zero-lattices}
The standard and nonstandard trace-zero lattices are
\[
 J_I^0\cong A_2\perp E_8^3,\qquad J_E^0=L_{26},
\]
where $L_{26}$ is even of rank $26$, determinant $3$, and minimum $4$.
\end{proposition}

\begin{proof}
The standard description is \eqref{eq:standard-trace-lattice}.  The other
assertions are \cref{prop:first-saturation,lem:residual-frame-existence}.
The isotope above identifies the latter with the Elkies--Gross lattice.
No lattice uniqueness theorem enters this proof.
\end{proof}

\subsection{The two theta series and embedding masses}

Now that the residual genus consists of one class, the full-genus identities
specialize to
\begin{equation}\label{eq:two-theta}
 \vartheta_{J_I}=E_{12}+\frac{432000}{691}\Delta,
 \qquad \vartheta_{J_E}=E_{12}-\frac{65520}{691}\Delta,
\end{equation}
\begin{equation}\label{eq:scalar-mass-average}
 \frac{91\vartheta_{J_I}+600\vartheta_{J_E}}{691}=E_{12},
\end{equation}
and
\begin{equation}\label{eq:G2-mass-average}
 91\Theta_I+600\Theta_E=546\mathscr E.
\end{equation}
Before uniqueness the expression
\[
 \frac{546\mathscr E-91\Theta_I}{600}
\]
was the normalized residual average; it is now the theta series of one
integral Albert class.  In particular, the three coefficients used in the
proof are
\[
\begin{array}{c|r|r|r}
 C&c_C(\Theta_I)&c_C(\Theta_E)&c_C(\mathscr E)\\\hline
 \Z^3&6&0&1\\
 \Z[t]/(t^3-t)&726&0&121\\
 \Z+2\Z^3&2262&14742&16577
\end{array}
\]
The last entry $14742=6\cdot2457$ counts ordered frames because the exotic
class has no content-$2$ lifts.  The standard coefficient is weighted and
must not be divided by $6$ to count its frames.

For a maximal totally real \'etale cubic order $C$, the Eisenstein
coefficient formula in weight $4$ gives
\[
 c_C(\mathscr E)=120^3\zeta_C(-3),
\]
since $\zeta_{\Q}(-3)=1/120$ and the split coefficient is $1$
\cite[Section~9 and the remark following Proposition~10.1]{GanGrossSavin}.
The latter remark establishes the coefficient formula unconditionally in
weight $4$.  Such an order is Gorenstein; hence
\eqref{eq:G2-embedding-coefficients} and \eqref{eq:G2-mass-average} give
\begin{equation}\label{eq:cubic-embedding-mass}
 \frac{N(C,J_I)}A+\frac{N(C,J_E)}B=\frac5{2016}\zeta_C(-3).
\end{equation}
This is an embedding mass, a sum of reciprocal stabilizers over embedding
orbits.  It need not consist of a single orbit.  A decomposable order
$\Z\times\cO_D$ can embed only in $J_I$, because the image of its rank-one
idempotent would give a trace-$1$ idempotent in $J_E$.

Finally, put $\Psi=(\Theta_I-\Theta_E)/6$.  Then
\[
 \mathscr E-\Psi=\frac{691}{546}\Theta_E.
\]
The integral lift coefficients in \eqref{eq:G2-lift-coefficient} yield the
coefficientwise congruence $\mathscr E\equiv\Psi\pmod{691}$ in
$\Z_{(691)}$ for the cubic-ring coefficients.  The difference direction has
a nonzero scalar boundary proportional to Ramanujan's $\Delta$
\cite[Corollary~2.5.1]{PollackMinimal}; it is not a cuspidal $G_2$ form.
There is also a Hecke check on this $\Delta$-direction.  The preceding
coefficient table gives
\[
 c_{C_0}(\Psi)=1,\qquad c_{R_4}(\Psi)=121,\qquad
 c_{C_2}(\Psi)=-2080.
\]
Applying the depth-zero formula of \cref{prop:G2-Eisenstein-coefficient}
to $\Psi$ gives
\[
 c_{C_0}\bigl(T_7(2)\Psi\bigr)
 =1+\frac3{16}\,121-\frac{2080}{128}=\frac{119}{16}.
\]
This agrees with the seven-dimensional character value predicted by the
$\Delta$-boundary.  Indeed, write its unitary elliptic parameters as
$\alpha_p\beta_p=1$ and
$\alpha_p+\beta_p=\tau(p)p^{-11/2}$, where $\tau(p)$ is Ramanujan's
coefficient.  After the modulus shift, the Levi parameters are
$p^{3/2}(\alpha_p,\beta_p)$, up to simultaneous inversion.  The restriction
formula in \cite[Sections~11--13]{GanGrossSavin} gives
\[
 1+p^3+p^{-3}+\tau(p)(p^{-4}+p^{-7}),
\]
which at $p=2$, with $\tau(2)=-24$, is again $119/16$.
Thus the split Fourier coefficient agrees with the scalar boundary.
These arithmetic consequences retain information about cubic subalgebras
beyond the scalar trace variable.

\subsection{A cubic field and its binary cubic form}\label{subsec:cubic-field49}

The split cubic rings used in the classification detect frames.  We now
revisit the irreducible discriminant-$49$ example of
\cite[Section~6]{ElkiesGrossCubic}, deriving its exceptional embedding count
from a calculation in the standard order and the two-class mass identity.
Consider
\[
 F=\Q(\alpha)=\Q(\zeta_7+\zeta_7^{-1}),\qquad
 \alpha=\zeta_7+\zeta_7^{-1},\qquad
 C=\cO_F=\Z[\alpha],
\]
where $\zeta_7$ is a primitive seventh root of unity.  The polynomial of
$\alpha$ is $p(t)=t^3+t^2-2t-1$, of discriminant $49$; the corresponding
binary cubic form is
\[
 f(u,v)=u^3+u^2v-2uv^2-v^3.
\]
Thus $C$ is a maximal totally real cubic order.  The cyclotomic description
and the special-value formulas used below are recalled in
\cite[Chapters~2--4]{Washington}.

Let $\chi$ be the cubic Dirichlet character modulo $7$ with
$\chi(3)=\zeta_3$, where $\zeta_3=e^{2\pi i/3}$.  Then
$\zeta_F(s)=\zeta(s)L(s,\chi)L(s,\bar\chi)$.  With
$B_4(t)=t^4-2t^3+t^2-1/30$, the Bernoulli formula gives
\[
 L(-3,\chi)=-\frac{7^3}{4}\sum_{a=1}^6\chi(a)B_4(a/7)
           =\frac{32-22\zeta_3}{7},
 \qquad \zeta_F(-3)=\frac{79}{210}.
\]
In the normalization of \eqref{eq:G2-normalization}, the Eisenstein
coefficient is consequently $c_C(\mathscr E)=4550400/7$.

The standard theta coefficient can be counted directly.  A unital
norm-preserving embedding of $C$ is determined by $x$, the image of
$\alpha$, and the required characteristic polynomial is equivalent to
\[
 \Tr(x)=-1,\qquad T(x,x)=5,\qquad N(x)=1.
\]
Conversely, these conditions give an embedding: Cayley--Hamilton gives
$p(x)=0$, and the irreducibility of $p$ identifies the power-associative
subalgebra $\Q[x]$ with $F$, with its cubic norm.  Integral squares and the
monic relation give $\Z[x]\subset J_I$.
Write $x$ with diagonal entries $a_i\in\Z$ and opposite off-diagonal
entries $c_i\in\cO$, and put $n_i=n(c_i)$.  Then
$\sum a_i^2+2\sum n_i=5$ and $\sum a_i=-1$.  Up to permutation, the only
diagonals are
\[
 (-1,0,0),\qquad (-1,-1,1),\qquad (-2,1,0).
\]
Their total off-diagonal norms are $2,1,0$, respectively, so in every case
at least one $c_i$ vanishes and the trilinear term in
\eqref{eq:triality-cubic} is zero.  The norm is $0$ or $2$ for diagonal $(-1,-1,1)$, and $0$ for $(-2,1,0)$.  For diagonal $(-1,0,0)$ the norm is
$n_1$, hence $n_1=1$ and $n_2+n_3=1$.  There are three positions for $-1$,
two choices for the other nonzero off-diagonal entry, and $240$ octonion
units for each nonzero entry.  Therefore
\[
 c_C(\Theta_I)=N(C,J_I)=3\cdot2\cdot240^2=345600.
\]
Substituting into \eqref{eq:cubic-embedding-mass} now gives
\[
 c_C(\Theta_E)=N(C,J_E)
 =B\left(\frac5{2016}\frac{79}{210}-\frac{345600}{A}\right)
 =539136.
\]
Thus $C$ embeds into both classes, unlike the decomposable orders above.
The genus identity recovers its exceptional count from the standard one.
Conjugacy of the exceptional embeddings is proved in
\cite[Section~6]{ElkiesGrossCubic}; it is not required here.

\section{A Venkov argument for the root-system dichotomy}\label{sec:venkov}

Let $J$ be any definite integral Albert algebra in the hyperspecial genus.
The argument uses class-supported scalar and linear theta series,
integrality, and the Peirce decomposition, but neither the Siegel--Weil
average nor coefficient saturation.  We retain the notation $J^0$ and
$R(J^0)$ from \cref{sec:dictionary}; the lattice $J^0$ is even of rank $26$
and determinant $3$.
A quaternionic counterpart is the rank-$3$ icosian Venkov theorem: two-place
harmonic separation shows that the three trace-minimal Hermitian shells
projectivize to $5$-designs \cite[Theorem~4.8]{HohnHallJanko}.

\subsection{Class-supported weighted exceptional theta series}

For $z\in J^0\otimes\C$, let
\[
 P_z(x)=T(x,z)=T\left(x-\frac{\Tr(x)}3E,z\right).
\]
These are the trace-free linear functions on $J\otimes\C$, the dual of the
$26$-dimensional $F_4$-module.  For such a linear form $P$, put
\begin{equation}\label{eq:weighted-exceptional-theta}
 \vartheta_{J,P}(\tau)
 =\sum_{\substack{0\ne x\in J,\ x\ge0\\x^{\#}=0}}
   \sigma_3(c(x))P(x)q^{\Tr(x)}.
\end{equation}
Only this linear case of the weighted correspondence is needed below.

\begin{proposition}[class-supported linear modularity]\label{prop:class-supported-theta}
For every $J$ in the hyperspecial genus and $z\in J^0\otimes\C$, the series
$\vartheta_{J,P_z}$ is a cusp form of weight $14$ on $\mathrm{SL}_2(\Z)$.
The assertion is independent of the number of classes in the genus.
\end{proposition}

\begin{proof}
Use Shan's differentiated kernel and weighted Fourier calculation
\cite[Theorems~5.1.2 and~5.1.4, Corollary~5.1.5]{ShanTheta}.
The norm-lattice transport and contragredient reindexing proved in
\cref{lem:scalar-theta-local} supply the calculation for an arbitrary class;
no choice from a known list of Albert orders is needed.  We then apply
stabilizer averaging before enumerating the class set.  Put
$\Gamma_J=\Aut(J)$ and let $P=P_z$.  Choose an algebraic modular form supported on the class of $J$ whose
value there is
\[
 \sum_{\gamma\in\Gamma_J}\gamma P,
 \qquad (\gamma P)(x)=P(\gamma^{-1}x).
\]
This value is $\Gamma_J$-invariant.  The single differential of the exceptional
theta kernel gives weight $14$; its class-local Fourier calculation yields,
up to a fixed nonzero normalization,
\[
 \frac1{|\Gamma_J|}
 \vartheta_{J,\sum_{\gamma\in\Gamma_J}\gamma P}
 =\vartheta_{J,P}.
\]
The equality follows by changing variables in each $\Gamma_J$-stable rank-one
shell.  The other double cosets do not contribute, so neither their number nor
their identities enter.  Since $P(0)=0$, the constant term vanishes; level one
has a single cusp, so the modular form is cuspidal.
\end{proof}

Finally,
\[
 S_{14}(\mathrm{SL}_2(\Z))
 =\Delta\,M_2(\mathrm{SL}_2(\Z))=0.
\]
Thus \cref{prop:class-supported-theta} gives the following first-moment identity.

\begin{corollary}[linear moment identity]\label{cor:F4-linear-moment}
For every integer $m\ge1$,
\begin{equation}\label{eq:F4-linear-moment}
 \sum_{\substack{x\in J,\ x\ge0,\ x^{\#}=0\\\Tr(x)=m}}
 \sigma_3(c(x))\left(x-\frac m3E\right)=0.
\end{equation}
\end{corollary}

\begin{proof}
Pair the left-hand side with every element of $J^0\otimes\C$ and read the coefficient of
$q^m$ in $\vartheta_{J,P_z}=0$.  The displayed vector has trace zero, so the trace pairing
is nondegenerate on the space in which it lies.
\end{proof}

\subsection{Trace-one idempotents and ordinary roots}

Set
\[
 X_1(J)=\{e\in J:e\ge0,\ e^{\#}=0,\ \Tr(e)=1\},
 \qquad a=a(J)=|X_1(J)|.
\]
Every element of $X_1(J)$ is a primitive idempotent by the spectral theorem for the
Euclidean Albert algebra.

\begin{lemma}[the trace-one alternative]\label{lem:trace-one-alternative}
One has
\[
 a=0\quad\text{or}\quad a=3.
\]
If $a=3$, the three elements of $X_1(J)$ form a Jordan frame.
\end{lemma}

\begin{proof}
The coefficient of $q$ in \eqref{eq:F4-linear-moment} gives
\begin{equation}\label{eq:sum-trace-one}
 \sum_{e\in X_1(J)}e=\frac a3E.
\end{equation}
For primitive idempotents $e,f$ in a Euclidean Jordan algebra,
\[
 0\le T(e,f)\le1.
\]
The pairing is integral on $J$, and equality $T(e,f)=1$ in Cauchy--Schwarz forces $e=f$.
Thus distinct elements of $X_1(J)$ are orthogonal in the trace pairing.
By \cref{lem:positive-orthogonality}, they are Jordan orthogonal.  Therefore the left-hand
side of \eqref{eq:sum-trace-one} is an idempotent.  Hence $(a/3)E$ is an idempotent, and therefore
$a/3\in\{0,1\}$.  In the second case the orthogonal idempotents sum to $E$.
\end{proof}

By \cref{lem:root-to-idempotent}, a root gives an element of $X_1(J)$.
Conversely, in the frame supplied by \cref{lem:trace-one-alternative}, each
$e_i-e_j$ is a root.  We have proved:

\begin{corollary}\label{cor:root-iff-frame}
The lattice $J^0$ is rootless if and only if $a=0$.  If it has a root, then $a=3$ and $J$
contains an integral Jordan frame.
\end{corollary}

\subsection{The second scalar coefficient}

Let $b$ be the number of primitive positive rank-one elements $x\in J$ with
$\Tr(x)=2$.  Apply the class-local scalar identity
\eqref{eq:theta-general-a}, with $a(J)=a$.
The nonprimitive rank-one elements of trace $2$ are exactly the elements $2e$ with
$e\in X_1(J)$.  Their content is $2$, and hence their weight is
$\sigma_3(2)=9$.  Comparing the coefficient of $q^2$ in
\eqref{eq:theta-general-a} gives
\[
 240(b+9a)
 =\frac{65520}{691}(2049+24)-24\cdot240a
 =196560-5760a.
\]
Thus
\begin{equation}\label{eq:F4-b-a}
 b=819-33a.
\end{equation}
Together with \cref{lem:trace-one-alternative}, this gives
\[
 \begin{array}{c|c|c}
 a&b&\text{ordinary roots}\\
 \hline
 0&819&\text{none},\\
 3&720&\text{present}.
 \end{array}
\]

\subsection{Peirce saturation in the rooted case}

Assume now that $a=3$, and write the integral frame as
\[
 e_1+e_2+e_3=E.
\]
Let $x$ be a primitive positive rank-one element of trace $2$, and put
\[
 d_i=T(x,e_i).
\]
Since $x=2p$ for a real primitive idempotent $p$, one has
\[
 0\le d_i\le2,
 \qquad d_i\in\Z,
 \qquad d_1+d_2+d_3=2.
\]
If $d_i=2$, equality in Cauchy--Schwarz gives $x=2e_i$, contrary to primitivity.  Hence
$(d_1,d_2,d_3)$ is a permutation of $(1,1,0)$.

Suppose, for example, that $d_k=0$, where $\{i,j,k\}=\{1,2,3\}$.  Orthogonality of the
positive elements $p$ and $e_k$ implies $p\circ e_k=0$.  The Peirce decomposition relative
to $e_k$ then gives
\begin{equation}\label{eq:trace-two-peirce-shape}
 x=e_i+e_j+y,
 \qquad y\in J\cap V_k=J\cap J_{ij}.
\end{equation}
As $T(x,x)=4$ and the Peirce decomposition is orthogonal,
\[
 T(y,y)=4-T(e_i+e_j,e_i+e_j)=2.
\]
Thus every primitive trace-$2$ rank-one element produces a root in exactly one
of the three rank-$8$ Peirce lattices.  For fixed $(i,j)$, the map
$x\mapsto y=x-e_i-e_j$ is injective.

The coefficient of $q^2$ in the vanishing linear theta series gives more.  For every
trace-free linear form $P=P_z$,
\[
 0=
 \sum_{\substack{x\text{ primitive},\ x\ge0,\ x^{\#}=0\\\Tr(x)=2}}P(x)
 +\sum_{e\in X_1(J)}\sigma_3(2)P(2e).
\]
The second sum is $18\sum_eP(e)=0$ by the coefficient of $q$.  Therefore the trace-free
part of the sum of the $720$ primitive elements vanishes, and
\begin{equation}\label{eq:sum-trace-two}
 \sum_{\substack{x\text{ primitive},\ x\ge0,\ x^{\#}=0\\\Tr(x)=2}}x
 =\frac{2\cdot720}{3}E=480E.
\end{equation}
Let $b_{ij}$ count the elements of the form \eqref{eq:trace-two-peirce-shape} supported on
$e_i,e_j$ and $V_k$.  Pairing \eqref{eq:sum-trace-two} with $e_1,e_2,e_3$ gives
\[
 b_{12}+b_{13}=480,
 \qquad
 b_{12}+b_{23}=480,
 \qquad
 b_{13}+b_{23}=480.
\]
These three equations yield
\begin{equation}\label{eq:240-each-peirce}
 b_{12}=b_{13}=b_{23}=240.
\end{equation}
Their sum also gives $b_{12}+b_{13}+b_{23}=720$, consistent with the total count.
By the injectivity above, each rank-$8$ Peirce lattice $J\cap V_k$ therefore
contains at least $240$ distinct roots.  The
norm-$2$ vectors of an integral positive-definite lattice form a simply laced root system,
and a simply laced root system of rank at most $8$ has at most $240$ roots, with equality
only for $E_8$.  Hence the three Peirce root systems are $E_8$.  The six vectors
$\pm(e_i-e_j)$ form an $A_2$ root system orthogonal to them.  We have therefore found a
full-rank sublattice
\[
 A_2\perp E_8\perp E_8\perp E_8\subseteq J^0.
\]
Its determinant is $3$, equal to $\det(J^0)$, so the inclusion has index one.

The determinant equality proves the rooted assertion of \cref{thm:venkov-intro};
\cref{cor:root-iff-frame} and \eqref{eq:F4-b-a} give the remaining assertions.
The argument used neither the number of Albert classes nor a mass identity.  Its
level-one hypothesis is essential: both the vanishing of $S_{14}$ and the
two-dimensionality of $M_{12}$ refer to $\mathrm{SL}_2(\Z)$.  The theta identities
explain the two possible sets of root and rank-one data; they do not alone
classify the integral cubic structures realizing those data.  That is the
additional role of the frame bounds and mass saturation.

\section{Finite geometry and identification of the groups}\label{sec:groups}

All class numbers, group orders, and frame-transitivity statements are now
known.  We first identify the automorphisms of the integral, lattice, and
incidence structures, then give the conventional group names.

\subsection{Frame symmetries and the full finite geometry}

\begin{corollary}[lifting the fixed-datum symmetries]\label{cor:shadow-equality}
For a primitive scale-$2$ frame $\mathcal F$ in $J_E$, restriction gives an isomorphism
\[
 \Aut(J_E)_{\mathcal F}\cong\Aut(\mathcal D_{\mathcal F}).
\]
Thus every automorphism of its oriented deletion datum is induced by an Albert
automorphism.  In particular,
\[
 1\longrightarrow C_2^5\longrightarrow\Aut(J_E)_{\mathcal F}
 \longrightarrow 2^6:\bigl((7:3)\times S_3\bigr)\longrightarrow1.
\]
The inverse image $P_{\mathcal F}$ of the normal $2^6$ has order $2^{11}$, and
\[
 \Aut(J_E)_{\mathcal F}/P_{\mathcal F}\cong(7:3)\times S_3.
\]
In particular, the frame stabilizer is solvable and induces the full $S_3$ on
the three elements of the frame.
\end{corollary}

\begin{proof}
The injection in \eqref{eq:stabilizer-injection} has source and target of order
$258048$.  The exact sequence is \eqref{eq:correct-deletion-sequence}, with
quotient \eqref{eq:eight-orbit-stabilizer}.  Its $S_3$ factor permutes the octads,
hence the frame elements.  Taking the inverse image of the normal $2^6$ proves
the remaining assertions.  No claim that $P_{\mathcal F}$ is elementary abelian,
or that the extension by it splits, is needed.
\end{proof}

Frame transitivity and the full $S_3$ action of \cref{cor:shadow-equality} imply
transitivity on incident point--frame pairs.  Thus, for $t\in\mathcal F$,
\begin{equation}\label{eq:point-flag-stabilizers}
 |\Aut(J_E)_t|=\frac{B}{819}=774144,\qquad
 |\Aut(J_E)_{t,\mathcal F}|=\frac{B}{819\cdot9}=86016.
\end{equation}
These numbers and the coarse frame-stabilizer structure are obtained without
using tables of subgroups.

\begin{corollary}[Albert and hexagon automorphisms]\label{cor:albert-hexagon}
Let $H$ be the generalized hexagon on $\mathcal R_2(J_E)$ and let
$\Aut^+(L_{26})$ preserve the nonzero discriminant class containing $X$.  Restriction
gives
\[
 \Aut(J_E)\cong\Aut^+(L_{26})\cong\Aut(H),\qquad
 \Aut(L_{26})\cong C_2\times\Aut(J_E).
\]
In particular, every incidence automorphism of $H$ extends to an automorphism of
the integral Albert algebra.
\end{corollary}

\begin{proof}
Albert automorphisms fix $E$ and preserve $\mathcal R_2(J_E)$, so restriction
embeds $\Aut(J_E)$ into $G^+:=\Aut^+(L_{26})$.  Every element of $G^+$ preserves
$X$ and its zero-sum triples, hence permutes the $2457$ lines.  It extends
uniquely to the index-three gluing $\Lambda_{J_E}$ by fixing $\xi=-E$.
For a line $\ell$, equivariance of deletion and the kernel argument of
\cref{lem:shadow-faithful} give an injection
$G^+_{\ell}\hookrightarrow\Aut(\mathcal D_{\mathcal F})$; that argument uses
only the marked trace lattice, not its cubic.  Thus
\[
 |G^+|\le2457\cdot258048=B=|\Aut(J_E)|,
\]
where the equality on the right is supplied by \cref{thm:main}.  The inclusion
is therefore an equality.  The central isometry $-1$ interchanges the two
nonzero discriminant classes and supplies the direct factor $C_2$ in
$\Aut(L_{26})$.

It remains to pass from lattice isometries to all incidence automorphisms.
The lattice--hexagon correspondence gives
$\Aut(H)\cong G^+$ \cite[Proposition~4.7]{Hoehn26}: distances determine the
centered Gram matrix, and the saturation argument there shows that $X$
generates $L_{26}^{\vee}$.  Every incidence automorphism therefore extends
uniquely to a lattice isometry preserving the chosen discriminant class.
Only this last step requires that correspondence, not lattice uniqueness or a
prior group-order calculation.
\end{proof}

This is a consequence of the mass as well as of the finite geometry.  The
orthogonal extension of an incidence automorphism is not assumed to preserve
the cubic tensor; equality of the two finite groups forces that preservation.
Similarly, complete Peirce--triality data become unique by frame transitivity,
with their rational tensor retained.  Neither conclusion follows from the
additive lattice alone.

\subsection{Cayley-plane rigidity}

\begin{remark}[Hoggar's geometric uniqueness assertion]\label{rem:hoggar}
Hoggar states that uniqueness of the generalized hexagon implies uniqueness of the
tight $819$-point Cayley-plane $5$-design, referring to a simple argument but not
supplying it there \cite[p.~93]{Hoggar1989}.  An incidence isomorphism determines the
centered Gram matrix and hence an orthogonal equivalence; the additional point is
that this equivalence must preserve the Albert product to lie in $F_4(\R)$.

The companion paper supplies this step explicitly for projective $3$-designs,
and hence for the tight $5$-design, by the third projective moment
\cite[Lemmas~5.2--5.3 and Theorem~5.6]{Hoehn26}.  For centered points $x_p=2p-\tfrac23E$ of a tight design
and $a,b\in J^0\otimes\R$, it gives
\[
 a\circ b-\frac{T(a,b)}3E
   =\frac1{24}\sum_p T(a,x_p)T(b,x_p)x_p.
\]
An angle-preserving bijection therefore preserves the traceless product; fixing $E$
recovers the full Albert product.  This makes the implication asserted by Hoggar
explicit for arbitrary realizations in the Cayley plane, whereas
\cref{cor:albert-hexagon} obtains the symmetry equality for the integral model
from the mass.  Neither the moment argument nor a prior geometric uniqueness
theorem is an input to mass saturation.

The quaternionic counterpart reconstructs the product on $\Her_3(\mathbb H)$
from the second and third moments and gives rigidity under
$\operatorname{PSp}(3)$ \cite[Lemma~7.3 and Theorem~7.4]{HohnHallJanko}.
There the distinction between angle-preserving and incidence symmetries is
essential: the Hall--Janko configuration has projective stabilizer $J_2$, whereas
its near octagon has automorphism group $J_2:2$.  The outer involution exchanges
the two golden-angle relations \cite[Corollary~7.6]{HohnHallJanko}.
\end{remark}

\subsection{Identification of the finite groups}

\begin{corollary}[group identifications]\label{cor:group-names}
The standard automorphism group is
\[
 \Aut(J_I)\cong\Tri(\cO)\rtimes S_3,
 \qquad \Tri(\cO)\text{ has shape }2^2\!\cdot\Omega_8^+(2).
\]
In the usual finite-group notation this is written
$2^2\!\cdot O_8^+(2)\!\cdot S_3$.  For the exotic order,
\[
 \Aut(J_E)\cong{}^3D_4(2):3.
\]
\end{corollary}

\begin{proof}
For $J_I$, the kernel of the action on the diagonal frame is the integral
triality group.  Permuting matrix indices gives an $S_3$ of Albert automorphisms
and splits this action.  Equality in \cref{prop:standard-bound} shows that
projection of $\Tri(\cO)$ onto the third orthogonal factor is onto $\SO(E_8)$,
with kernel of order $2$.

The quadratic space $E_8/2E_8$, with $q(\bar x)=x^2/2\pmod2$, has plus type.
The classical reduction of the $E_8$ Weyl group gives
\[
 \SO(E_8)/\{\pm1\}\cong\Omega_8^+(2);
\]
see \cite[Chapter~4]{ConwaySloane} and \cite[Example~7.3]{Conrad}.
The inverse image of $\{\pm1\}$ in $\Tri(\cO)$ is the four-element group of
scalar triples $(\varepsilon_1,\varepsilon_2,\varepsilon_3)$ with
$\varepsilon_i\in\{\pm1\}$ and
$\varepsilon_1\varepsilon_2\varepsilon_3=1$.  This proves the stated central
extension shape.  The intrinsic identification is with $\Tri(\cO)\rtimes S_3$;
the dotted notation summarizes its factors.  Here $\Omega_8^+(2)$ is the simple
group, often denoted $O_8^+(2)$ in finite-group notation, not the full isometry
group of the quadratic space over $\F_2$.

For $J_E$, \cref{cor:albert-hexagon} identifies the arithmetic automorphism group
with $\Aut(H)$.  The uniqueness theorem for the generalized hexagon of order
$(2,8)$ identifies $H$ with the point--line dual of the twisted triality
hexagon $T(8,2)$, whose order is $(8,2)$ \cite{CohenTits}; alternatively one
may use the independent Niemeier
reconstruction in \cite[Sections~4 and~6]{Hoehn26}.  This uniqueness is used here
only to name the group, after its order has been proved.
Steinberg's twisted $D_4$ construction and Tits' triality construction give a
faithful collineation action
\[
 {}^3D_4(2):3\;\leq\;\Aut(H)
\]
\cite{Steinberg1959,Tits1959,TitsBuildings,CarterLie}.  The standard order formula
\[
 |{}^3D_4(q)|=q^{12}(q^8+q^4+1)(q^6-1)(q^2-1)
\]
gives $|{}^3D_4(2):3|=B$.  Since $|\Aut(H)|=B$ by the mass argument, the
inclusion is an equality.  Neither identification uses an ATLAS entry for the
full collineation group or its order.
\end{proof}

For an a posteriori check, a Borel subgroup $\mathsf B$ of ${}^3D_4(2)$
has order $2^{12}(2-1)(2^3-1)=28672$.  In the point--line convention of $H$,
its indices in the point and line parabolics are $9$ and $3$, respectively
\cite{TitsBuildings,CarterLie}.  The standard order-three outer automorphism normalizes
$\mathsf B$ and both parabolics.  Hence the semilinear stabilizers have
orders
\[
 \begin{aligned}
 |\Stab(x)|&=3\cdot9\cdot|\mathsf B|=774144,\\
 |\Stab(\ell)|&=3\cdot3\cdot|\mathsf B|=258048,\\
 |\Stab(x,\ell)|&=3\cdot|\mathsf B|=86016\qquad(x\in\ell).
 \end{aligned}
\]
These are exactly the point, frame, and flag stabilizer orders obtained
before the group identification.

\appendix

\section{Complete Peirce--triality gluing data}\label{sec:PTcode}

The gluing code records the full framed Albert order, not merely its trace lattice.
We give the reconstruction statement used in \cref{thm:main}, with the separation
condition needed to recover the specified Peirce intersections.

\subsection{The separated lattice}

For an integral Albert algebra $J$ with a rational frame, use
\eqref{eq:Peirce} and put
\begin{equation}\label{eq:pieces}
 A_{\mathrm{diag}}=J\cap D,\qquad M_i=J\cap V_i,\qquad
 L_{\mathrm{sep}}=A_{\mathrm{diag}}\perp M_1\perp M_2\perp M_3.
\end{equation}
This is a full sublattice of $J$.  With duals taken for $T$,
\[
 L_{\mathrm{sep}}\subset J\subset L_{\mathrm{sep}}^\vee,\qquad C_{\mathrm{gl}}=J/L_{\mathrm{sep}}\subset D(L_{\mathrm{sep}})=L_{\mathrm{sep}}^\vee/L_{\mathrm{sep}}.
\]
The rational forms $q_i$ and the triality tensor $\tau$ are retained together with
this finite additive datum.  Write $\mathcal C(J,\mathcal F)$ for the resulting
complete datum.  It distinguishes a cubic order from its quadratic lattice.  Since the summands are orthogonal,
\[
 D(L_{\mathrm{sep}})=D(A_{\mathrm{diag}})\oplus D(M_1)\oplus D(M_2)\oplus D(M_3).
\]
Because the summands in \eqref{eq:pieces} are the actual intersections with their
rational spaces, the glue satisfies
\begin{equation}\label{eq:separated-code}
 C_{\mathrm{gl}}\cap D(A_{\mathrm{diag}})=0,\qquad C_{\mathrm{gl}}\cap D(M_i)=0\quad(1\le i\le3),
\end{equation}
where each discriminant group is embedded as its own summand.

\begin{definition}[admissible complete gluing datum]\label{def:PTcode}
A Peirce--triality gluing datum consists of the rational framed Albert package
\eqref{eq:triality-cubic}, full lattices $A_{\mathrm{diag}}\subset D$ and
$M_i\subset V_i$, and a subgroup $C_{\mathrm{gl}}\subset D(L_{\mathrm{sep}})$, where $L_{\mathrm{sep}}$ is integral for $T$.
It is \emph{admissible} if $C_{\mathrm{gl}}$ is isotropic for the discriminant bilinear form,
satisfies \eqref{eq:separated-code}, and its inverse image
\begin{equation}\label{eq:reconstructed-order}
 J(C_{\mathrm{gl}})=\{x\in L_{\mathrm{sep}}^\vee:x+L_{\mathrm{sep}}\in C_{\mathrm{gl}}\}
\end{equation}
contains $E$ and is an Albert algebra over $\Z$ for the given rational cubic
structure.  Equivalently, the norm and adjoint restrict as integral polynomial laws,
and the resulting local cubic structure is the split Albert algebra over every
$\Z_p$.  A specified scale is recovered as the least $m>0$ with $me_i\in J(C_{\mathrm{gl}})$.
\end{definition}

The polynomial-law and local conditions are part of admissibility; additive
isotropy alone does not imply them, especially at $2$.  Similarly, without
\eqref{eq:separated-code} the inverse image could enlarge one of the individual
Peirce lattices.  No general classification of admissible tensors or of their
congruences is assumed.

An isomorphism of data is an isomorphism of the rational framed Albert packages,
allowing a simultaneous permutation of the three frame elements and Peirce indices,
which carries the four lattices and $C_{\mathrm{gl}}$ to their counterparts.  Thus it preserves
the identity, cubic expression, and full glue.  For ordered frames the permutation
is omitted.

\begin{proposition}[reconstruction]\label{thm:reconstruction}
Taking the Peirce intersections and their quotient code, and taking the inverse
image \eqref{eq:reconstructed-order}, are inverse constructions on framed integral
Albert algebras and admissible complete gluing data.  They identify isomorphisms;
in particular,
\[
 \Aut(J,\mathcal F)\cong\Aut(\mathcal C(J,\mathcal F)).
\]
\end{proposition}

\begin{proof}
Starting from $J$, the preimage of $J/L_{\mathrm{sep}}$ in $L_{\mathrm{sep}}^\vee$ is $J$.  Its integrality and
local Albert structure are inherited.  Conversely, admissibility makes $J(C_{\mathrm{gl}})$ an
integral Albert algebra.  Its intersection with a rational Peirce summand can
exceed the prescribed lattice only if $C_{\mathrm{gl}}$ meets the corresponding discriminant
summand nontrivially.  Condition~\eqref{eq:separated-code} excludes this, so all four
intersections, and hence $L_{\mathrm{sep}}$ and $C_{\mathrm{gl}}$, are recovered.

An isomorphism of data preserves the rational norm and identity, hence the Albert
structure, and carries the preimage of $C_{\mathrm{gl}}$ to that of the target code.  Conversely,
a frame-preserving Albert isomorphism preserves the rational Peirce spaces and
their integral intersections, and therefore induces an isomorphism of data.
\end{proof}

\subsection{The standard and extremal diagonal orders}

For the standard diagonal frame,
\[
 A_{\mathrm{diag}}=\Z^3,\qquad M_1=M_2=M_3=\cO,\qquad C_{\mathrm{gl}}=0,
\]
with the octonion norm and its triality tensor.  The whole order is already
separated, and its frame automorphisms are the integral trialities and index
permutations used in \cref{prop:standard-bound}.

For a scale-$2$ frame in an extremal order, the diagonal intersection is exactly
\begin{equation}\label{eq:diagonal-order}
 A_{\mathrm{diag},2}
 =\Z E+\Z(2e_1)+\Z(2e_2)
 =\{(a_1,a_2,a_3)\in\Z^3:a_1\equiv a_2\equiv a_3\pmod2\}.
\end{equation}
To see this, the rational eigenvalues of an integral diagonal element are roots of
its monic integral characteristic polynomial, hence are integers.  Thus the
intersection lies in $\Z^3$ and contains the displayed parity order.  Any proper
intermediate lattice contains some $e_i$: modulo the parity order the three
nonzero classes are represented by $e_1,e_2,e_3$.  This contradicts extremality.
The diagonal order is the cubic order $\Z+2\Z^3$, of index $4$ in $\Z^3$ and
trace discriminant $16$.  It is not Gorenstein at $2$.  Indeed, let $u$ and $v$
be the classes of $2e_1$ and $2e_2$ in
$A_{\mathrm{diag},2}/2A_{\mathrm{diag},2}$.  Then
\[
 A_{\mathrm{diag},2}/2A_{\mathrm{diag},2}
 \cong\F_2[u,v]/(u^2,uv,v^2),
\]
whose socle is spanned by $u,v$ and has dimension $2$.  This explains why the
Gorenstein coefficient formula in
\eqref{eq:G2-embedding-coefficients} cannot simply be applied to count these
scaled frames.

\section*{Acknowledgements}

The author is grateful to Robert Langlands and to the late G\"unter Harder for
discussions about Tamagawa numbers during his graduate studies.  He also thanks
Gabriele Nebe for helpful discussions, many years ago, on the topics of this paper.

% Alpha-style citation labels; entries are ordered alphabetically by author.
% The bibliography is kept inline so this source is self-contained.

\end{document}